\documentclass[10pt,leqno]{amsart}
\usepackage{graphicx}
\usepackage{indentfirst,csquotes,amsthm, amscd, amsfonts, amssymb,setspace}

\usepackage{tikz,tikz-cd}
\usepackage{geometry}   
\usepackage{pb-xy}
\usepackage{mathtools}
\usepackage[shortlabels]{enumitem}
\usepackage{float}
\usepackage{hyperref}
\newcommand{\newuparrow}{{{\rlap{$\ $}\hbox{$\uparrow$}}}}%
\newcommand{\newdownarrow}{{{\rlap{$\ $}\hbox{$\downarrow$}}}}%

 \newcommand{\tbigcup}{\mathop{\textstyle \bigcup}}%%
 \newcommand{\tbigvee}{\mathop{\textstyle \bigvee}}%%
 \newcommand{\tbigwedge}{\mathop{\textstyle \bigwedge}}%%

\newcommand{\set}[1]{\{\,#1\,\}}

\newcommand{\alg}[1]{\ensuremath{\mathfrak{#1}}} % Algebra
\newcommand{\cat}[1]{\ensuremath{\mathsf{#1}}} % Category
\theoremstyle{plain}
          \newtheorem{theorem}{Theorem}[section]
          \newtheorem{lemma}[theorem]{Lemma}
          \newtheorem{proposition}[theorem]{Proposition}
	\newtheorem{corollary}[theorem]{Corollary}
	\newtheorem{result}[theorem]{Result}

        \theoremstyle{definition}
          \newtheorem{example}[theorem]{Example}
          \newtheorem{examples}[theorem]{Examples}
          \newtheorem{remark}[theorem]{Remark}
\newtheorem{question}[theorem]{Question}
\newtheorem{definition}[theorem]{Definition}
\newtheorem{remarks}[theorem]{Remarks}

\newtheorem{comments}[theorem]{Comments}
\numberwithin{equation}{section}
\newcounter{num}[section]

\newcommand{\donotbreakdash}[1]{#1\nobreakdash-\hspace{0pt}}
\tikzset{every picture/.style={line width=0.7pt}}
\newcommand{\mycircle}[1]
         {\draw[fill=white,line width=0.7pt] (#1) circle[radius=2pt]}

\begin{document}
\title{Unitally nondistributive quantales} %%%%%%%%%%%%
\author[J. Guti\'errez Garc{\'\i}a]{Javier Guti\'errez Garc{\'\i}a}
\thanks{The first named author acknowledges support from the Basque Government (grant IT1483-22). }
\date{\today}
\address{Departamento de Matem\'aticas, Universidad del Pa\'{\i}s Vasco (UPV/EHU), 48080, Bilbao, SPAIN}
\email{javier.gutierrezgarcia@ehu.eus}
\author[U. H\"ohle]{Ulrich H\"ohle}
\date{\today}
\address{Fakult\"{a}t f\"{u}r Mathematik und Naturwissenchaften, Bergische Universit\"{a}t, D-42097, Wuppertal, GERMANY}
\email{uhoehle@uni-wuppertal.de}
\maketitle

{\let\thefootnote\relax
 }%%%%%%%%%%

\begin{abstract}
Unitally nondistributive quantales are unital quantales such that the unit is approximable by the totally below relation and does not meet-distribute over arbitrary joins. It is shown that the underlying nondistributive complete lattice contains at least $7$ elements.
Moreover, under mild conditions, every quantale has an extension to a unitally nondistributive quantale by the addition of an isolated unit. As a byproduct of this construction we prove that there exist $30$ non-isomorphic, unitally nondistributive quantales on the set of $7$ elements.
\smallskip

\end{abstract} %%%%%%%%%
\bigskip

\section*{Introduction} 
In this paper we present a new class of quantales arising from a question posed by Dirk~Hofmann and Maria Manuel Clementino during the XIV Portuguese Category Seminar celebrated at Coimbra, Portugal in  October 2023 (see \cite{CHT}). The essence of this question can be formulated as follows. Let $\alg{Q}$ be a unital quantale. What impact has the approximability of the unit of $\alg{Q}$ by the totally below relation on the distributivity of the underlying lattice of $\alg{Q}$. In this context it is important to point out that under the existence of a dualizing element and the distributivity of the quantale multiplication over nonempty meets  the approximability of the unit of $\alg{Q}$ by the totally below relation is equivalent to the completely distributivity of the underling lattice (cf.\ Proposition~\ref{propB}). Therefore the question arises what happens in the absence of a dualizing element or the missing distributivity of the quantale multiplication over nonempty meets.

The aim of this paper is to show that the approximability of the unit of a quantale by the totally below relation is completely unrelated to any kind of distributivity of the underlying lattice. For this purpose we introduce the class of unitally nondistributive quantales. 
As a first step we investigate their underlying complete  lattice and encounter a new class of nondistributive lattices, in which the approximability of the unit by the totally below relation entails a stronger form of nondistributivity being called here strict nondistributivity (cf.\ Definition~\ref{definition1}). 
The investigation of this notion leads us to the first important result  (cf.\ Proposition~\ref{newthm1}), 
which provides a necessary condition for a complete lattice 
to be strictly nondistributive in terms of $4$ specific complete sublattices. %, and certainly resembles Birkhoff's characterization of nondistributive lattices.
  Further, every nondistributive complete  lattice has an extension to a strictly nondistributive one (cf.\ Lemma~\ref{newnewlemma1.1C}). 
In a next step we present the second important result of this paper  (cf.\ Theorem~\ref{newthm2}), a construction showing how, under mild conditions, quantales on nondistributives lattices can be extended to unitally nondistributive quantales by addition of an isolated unit.
 On this background, it is not difficult to see that every unitally nondistributive quantale contains at least $7$ elements (cf. Corollary~\ref{newcor1}). This is the reason why we provide the classification of all unitally nondistributive quantales on the set of $7$ elements and obtain that on the set of $7$ elements there exist exactly $30$ non-isomorphic, unitally nondistributive quantales (cf. Corollary~\ref{newcor2} and Examples \ref{example2} and \ref{example3}). All of them could be seen as a simple counterexample concerning the question motivating our study (cf.\ Subsection~\ref{subsec:3.1}). Two classes of  some more advanced counterexamples dealing with at least $8$ elements appear in Subsection~\ref{subsec:3.2} completing the picture given by Proposition~\ref{newthm1}.

\section{Preliminaries on quantales and motivation}
\label{sec0:}

Let $\cat{Sup}$ be the symmetric, monoidal closed category of complete lattices and join-preserving maps. A \emph{quantale} is a semigroup in $\cat{Sup}$, a \emph{unital} quantale is a monoid in $\cat{Sup}$, and an \emph{integral} quantale is a unital quantale such that the unit coincides with the upper universal bound $\top$ (cf.\ \cite{EGHK}). Due to the tensor product in $\cat{Sup}$, every quantale $\alg{Q}$ can be identified with a complete lattice provided with a semigroup operation $\ast$ which is join-preserving in each variable separately, i.e.\ $\bigl(\tbigvee A\bigr)\ast \beta=\tbigvee_{\alpha\in A} \alpha\ast \beta$ and $\beta\ast \bigl(\tbigvee A\bigr)=\tbigvee_{\alpha\in A}\beta\ast \alpha, $ for each $A\subseteq \alg{Q}$ and $\beta\in \alg{Q}$.
The right- and  left-implication are given by 
\[\alpha\searrow \beta=\tbigvee\set{\gamma\in \alg{Q}\mid\alpha\ast \gamma\le \beta}\quad \text{and}\quad \beta\swarrow \alpha=\tbigvee\set{\gamma\in \alg{Q}\mid \gamma\ast \alpha\le \beta},\qquad \alpha,\beta\in \alg{Q}.\]
Let $\alg{Q}$ be a quantale.
An element $\delta\in \alg{Q}$ is \emph{dualizing} if for all $\alpha\in \alg{Q}$ the relation ${\delta\swarrow(\alpha\searrow\delta)}=\alpha={(\delta\swarrow\alpha)}\searrow\delta$ holds. If $\alg{Q}$ has a dualizing element, then it is unital. $\alg{Q}$  is \emph{semi-unital} if $\alpha\le \top\ast \alpha$ and $\alpha\le \alpha\ast \top$ hold for all $\alpha\in \alg{Q}$. An element $\alpha\in \alg{Q}$ is \emph{two-sided} if $\top\ast \alpha\le \alpha$ and $\alpha\ast \top\le \alpha$, and $\alg{Q}$ is \emph{two-sided} if every element of $\alg{Q}$ is two-sided. Every integral quantale is two-sided, but not vice versa.

Let $L$ be a complete lattice and $\vartriangleleft$ be the totally below relation on $L$, i.e.\ given $\alpha,\beta\in L$, we write $\beta\vartriangleleft \alpha$ if for any subset $A\subseteq L$ with $\alpha\leq \tbigvee A$ there is a $\gamma\in A$ such that $\beta\le\gamma$. We recall that a complete lattice is \emph{completely distributive} if every $\alpha\in L\setminus\set{\bot}$ is \emph{\donotbreakdash{$\vartriangleleft$}approximable}, i.e. $\alpha\le \tbigvee\set{\beta\in L\mid \beta \vartriangleleft \alpha}$.
Note that if $\alpha\vartriangleleft \alpha$, i.e.\ $\alpha$ is \emph{completely join-prime} (cf.\ \cite{DP}), then $\alpha$ is \donotbreakdash{$\vartriangleleft$}approximable.

Before we proceed we first show that under certain algebraic properties of a  unital quantale the \donotbreakdash{$\vartriangleleft$}approximability of its unit implies the complete distributivity of its underlying complete lattice.

\begin{proposition}\label{propB} Let $\alg{Q}=(\alg{Q},\ast,e)$
be a unital quantale with a dualizing element such that the
quantale multiplication is distributive over nonempty meets. If the unit $e$ is \donotbreakdash{$\vartriangleleft$}approximable, then the underlying lattice of $\alg{Q}$ is completely
distributive.
\end{proposition}

\begin{proof}
Let $\alpha\in\alg{Q}\setminus\set{\bot}$, $\delta$ be a dualizing element of $\alg{Q}$ and $e$ be the unit of $\alg{Q}$. Further, let $\beta\in\alg{Q}$ such that $\beta \vartriangleleft e$ and $\varnothing\ne A\subseteq \alg{Q}$ with $\alpha \le \tbigvee A$. 
 If the quantale multiplication distributes over nonempty
meets, then
\begin{align*}
\beta\vartriangleleft e&\le\alpha\searrow \bigl(\tbigvee A\bigr)
=\alpha\searrow\bigl(
\bigl(\delta\swarrow\bigl(\tbigvee A\bigr)\bigr) \searrow
\delta\bigr)=\alpha\searrow\bigl(
\bigl(\tbigwedge\limits_{\gamma\in
A}\bigl(\delta\swarrow\gamma\bigr)\bigr) \searrow \delta\bigr)\\[-2pt]
&=
\bigl(\bigl(\tbigwedge_{\gamma\in A} (\delta\swarrow
\gamma)\bigr)\ast \alpha\bigr)\searrow\delta
=\tbigvee_{\gamma\in A} \bigl(\bigl((\delta\swarrow
\gamma)\ast \alpha\bigr)\searrow \delta\bigr)
 = \tbigvee_{\gamma\in A} (\alpha\searrow \gamma).
\end{align*}
Hence there exists $\gamma\in A$ such that $\beta \le \alpha \searrow \gamma$, i.e.\
$\alpha \ast \beta \le \gamma$ and it follows that $(\alpha\ast\beta)\vartriangleleft \alpha$. 
Since $e$ is  \donotbreakdash{$\vartriangleleft$}approximable, we obtain
\[\alpha\le\alpha\ast\bigl(\tbigvee\set{\beta\in \alg{Q}\mid \beta \vartriangleleft e}\bigr)= \tbigvee\set{\alpha\ast\beta\mid \beta\in
\alg{Q}\text{ and }\beta \vartriangleleft e}\le
\tbigvee\set{\gamma\in\alg{Q}\mid \gamma \vartriangleleft
\alpha}\]
We conclude that $\alpha$ is \donotbreakdash{$\vartriangleleft$}approximable and consequently  the underlying lattice of $\alg{Q}$ is completely
distributive.
\end{proof}

If in a unital quantale we abandon the existence of a dualizing element 
or the property that
the quantale multiplication is distributive over nonempty meets,
then Proposition~\ref{propB} suggests the question which
relationship exists between the \donotbreakdash{$\vartriangleleft$}approximability of the the unit $e$ and
distributivity properties of the underlying lattice. As we already pointed out in the introduction, this question arose also during the XIV Portuguese Category Seminar celebrated at Coimbra, Portugal in  October 2023. More precisely, Dirk Hofmann and Maria Manuel Clementino asked the following:

\begin{question}\label{question1}
If the unit of a unital quantale  is \donotbreakdash{$\vartriangleleft$}approximable, does this property imply
%is it true 
that  then the unit \donotbreakdash{$\wedge$}distributes over arbitrary joins?
\end{question}
 
The motivation of this article is to answer this question negatively by proving that there exist a large class of unital quantales $\alg{Q}=(\alg{Q},\ast,e)$ such that the unit $e$ is
\donotbreakdash{$\vartriangleleft$}approximable and there exists a subset $A$ of $\alg{Q}$ satisfying the property $e\wedge\bigl(\tbigvee A\bigr)\not\le \tbigvee_{\alpha\in A}
(e\wedge \alpha)$.
We shall call these quantales \emph{unitally nondistributive}.
It follows immediately that every unitally nondistributive quantale is non-integral and its
underlying lattice is far from being a frame. However, it is interesting to see that every
quantale induced by an arbitrary group can be extended to a
unitally nondistributive quantale (cf. Example~\ref{examples1}\,(2)).

%%%%%%%%%%%%%%%%%%%%%%%%%%%%%%%%%%%%%%%%%%%%%%%%%%%%%%%%%%%%%%%%%%%%%%%%%%%%%%%%%%%%%%%%%%%%%%%%%%%%%%%%%%%%%%%%%%%%%%%%%%%%%%%%%%
%%%%%%%%%%%%%%%%%%%%%%%%%%%%%%%

\section{Strictly nondistributive lattices}
 \label{sec:1}
First we recall some terminology from lattice theory. Let $L$ be
a lattice and $\alpha,\beta,\gamma\in L$. $\alpha$ is \emph{join-irreducible} if $\alpha\ne 0$ and $\alpha = \beta\vee\gamma$ implies
that $\alpha = \beta$ or $\alpha =\gamma$, and $\alpha$ and $\beta$ are
\emph{incomparable}, if $\alpha\not\le \beta$ and $\beta\not\le
\alpha$. 
We will also use the following notation: 
\[\newdownarrow\alpha=\set{\beta\in L\mid\beta\le
\alpha}\quad\text{ and }\quad\newuparrow\alpha={\set{\beta\in
L\mid\alpha\le \beta}}.\]
An element $\alpha\in
L\setminus\{\bot,\top\}$ is \emph{isolated}\footnote{The property being isolated means also to be doubly
(completely) irreducible.} in $L$ if there
exist two elements $\alpha^-\le\alpha$ and $\alpha^+\ge\alpha$ of $L$ such that
following properties hold:
\[(\newdownarrow\alpha)\setminus\set{\alpha}=%\set{\gamma\in L\mid\gamma\le \alpha,\  \alpha\neq\gamma}=
\newdownarrow\alpha^-\text{ and}\quad (\newuparrow\alpha)\setminus\set{\alpha}=%\set{\gamma\in L\mid \alpha \le \gamma,\ \alpha \neq\gamma}=
\newuparrow\alpha^+.\]
It is easy to see that monomorphisms in the category $\cat{Sup}$ are join-preserving and injective maps. 

Motivated by the concept of quantales (cf.\ Section~\ref{sec0:}) the categorical framework in this paper will be always $\cat{Sup}$. Hence a \emph{complete sublattice} $S$ of a complete lattice $L$ is a complete lattice $S$ such that the inclusion map  $S \xhookrightarrow{\,\,\,} L$ is  join-preserving. Therefore joins in $S$ are joins in the sense of $L$, but not meets.

In what follows we are interested in nondistributive lattices. Recall that a  lattice $L$ is \emph{nondistributive} if  and only if  there exist three elements $\alpha,\beta,\gamma\in L$ such that 
 \stepcounter{num}
\begin{equation}\label{nondistributive}\gamma\wedge (\alpha\vee \beta)\not\le
(\gamma\wedge \alpha)\vee (\gamma\wedge \beta).
\end{equation} 
It is well know that without loss of generality we can always choose $\gamma$ in (\ref{nondistributive}) satisfying the property $\gamma\le \alpha\vee \beta$. But then (\ref{nondistributive}) implies that $\gamma$ fails to be \donotbreakdash{$\vartriangleleft$}approximable. Typical examples of nondistributive lattices provided with this structure are the diamond with three atoms or the pentagon (see $\mathsf{M}_3$ and $\mathsf{N}_5$ in Figure~\ref{fig:N5 and M3}). 
\pagebreak

Since our main interest is to provide an answer to Question~\ref{question1} and to investigate nondistributive, complete lattices, in which certain elements are  \donotbreakdash{$\vartriangleleft$}approximable, we introduce the following terminology. 

\begin{definition}\label{definition1}
A  lattice $L$ is \emph{strictly nondistributive} if  there exist $\alpha,\beta,x\in L$ such that 
 \stepcounter{num}
\begin{equation}\label{strictly nondistributive}x\wedge (\alpha\vee \beta)\not\le
(x\wedge \alpha)\vee (x\wedge \beta)\quad\text{and}\quad x\not\le \alpha\vee \beta.
\end{equation} 
\end{definition}

As we have seen above, not every nondistributive lattice is strictly nondistributive.
 In fact, if  in $\mathsf{M}_3$ or $\mathsf{N}_5$ we choose $\alpha,\beta,\gamma\in L$ satisfying (\ref{nondistributive}), 
then $\alpha\vee \beta$ coincides necessarily with the universal upper bound.  But  we will show now that every nondistributive lattice can be \emph{extended} to a strictly nondistributive lattice.
  This construction will be  based on  the addition of an isolated element and will play a significant role in what follows. 
Due to the lattice structure this isolated element $x$ will always be completely join-prime, i.e. $x\vartriangleleft x$.
For this purpose we  first fix some further terminology.

Let us recall  that we are working in $\cat{Sup}$.  
Then a complete sublattice $S$ of
a complete lattice $L$ is called a \emph{pentagon}, respectively a \emph{diamond with $3$ atoms},
 if $S$ is isomorphic to  $\mathsf{N}_5$, respectively to $\mathsf{M}_3$. 
They are \emph{typical} and 
\emph{minimal} among nondistributive lattices in the sense
of their cardinality. 
Consequently any nondistributive lattice must have at least $5$ elements.
We are also interested in another two nondistributive lattices with $6$ and $7$ elements which we denote by $\mathsf{L}_6$ and $\mathsf{L}_7$, respectively,  and whose Hasse diagrams are given in Figure~\ref{fig:L6 and L7} below. 

A complete lattice $M$ is an \emph{extension} of a complete lattice $L$ if the inclusion map
$L\xhookrightarrow{} M$ is arbitrary join-preserving  and nonempty meet-preserving. In particular $L$ is always a complete sublattice of $M$. 

The next lemma describes a special extension of a complete lattice by two different elements, one of them being isolated. 
In particular, this extension will always transfer a nondistributive lattice into a \emph{strictly} nondistributive one. 
Note also that since we are working in $\cat{Sup}$, we assume here completeness, but it is evident that the basis of this lemma  requires only arbitrary lattices.

\begin{lemma}\label{newnewlemma1.1C} Let $L$ be a complete  lattice 
%\footnote{Since we are working in $\cat{Sup}$, we assume here completeness. It is evident that the basis of this lemma  requires only  arbitrary lattices.} 
and $\overline{\top}$ and $x$ be two different elements satisfying the condition $L\cap\set{\overline{\top},x}=\varnothing$. Then for every $\gamma \in L$ there exists a unique complete lattice-structure on  $\overline{L}^\gamma=L\cup\set{\overline{\top},x}$ satisfying the following conditions\textup:
\begin{enumerate}[label=\textup{(\roman*)},align=right,leftmargin=0pt,labelwidth=15pt,itemindent=25pt,labelsep=5pt,topsep=5pt,itemsep=3pt]
\item \label{(i)} $\overline{L}^\gamma$ is an extension of $L$.
\item \label{(ii)} The element $\overline{\top}$ is the universal upper bound of  $\overline{L}^\gamma$.
\item \label{(iii)} The element $x$ is isolated and satisfies the conditions $x^-=\gamma$ and $x^+=\overline{\top}$.
\end{enumerate}
Moreover, if there exist $\alpha,\beta\in L$ such that $\set{\alpha,\beta,\gamma}$ satisfies {\rm (\ref{nondistributive})}, then  $\set{\alpha,\beta,x}$ satisfies {\rm (\ref{strictly nondistributive})} in $\overline{L}^\gamma$, i.e.\ $\overline{L}^\gamma$ is strictly nondistributive.
\end{lemma}

\begin{proof} (Uniqueness) With regard to \ref{(i)} and \ref{(ii)} it is sufficient to investigate the order relation between the element $x$ and all other elements of $L$. It is easy to see that \ref{(iii)} is equivalent to the following properties in  $\overline{L}^\gamma$:
\begin{enumerate}[label=\textup{(\roman*)},align=right,leftmargin=0pt,labelwidth=15pt,itemindent=25pt,labelsep=5pt,topsep=5pt,itemsep=3pt,start=4]
\item \label{(iv)} if $\alpha\le \gamma$ in $L$, then $\alpha\le x$ in $\overline{L}^\gamma$,
\item \label{(v)} if $\alpha \not\le \gamma$  in $L$, then $\alpha$  is incomparable with $x$ in $\overline{L}^\gamma$.
\end{enumerate}

\noindent 
(Existence)  Obviously, the partial order on $L$  can be extended to $\overline{L}^\gamma$ by using the relations \ref{(iv)} and \ref{(v)} and fixing $\overline{\top}$ as universal upper bound.  Then \ref{(i)} --- \ref{(iii)} are satisfied. In order to confirm the completeness of $\overline{L}^\gamma$ it is sufficient to observe that for all $\alpha\in L$ the following relations hold:
{\begin{enumerate}[label=\textup{(\roman*)},align=right,leftmargin=0pt,labelwidth=15pt,itemindent=25pt,labelsep=5pt,topsep=5pt,itemsep=3pt,start=6]
\item \label{(vi)} $x\wedge \alpha=\gamma\wedge \alpha$, 
\item \label{(vii)} if $\alpha\le \gamma$ in $L$ then $\alpha\vee x=x$,
\item \label{(viii)} if $\alpha \not\le \gamma$  in $L$ then $\alpha\vee x=\overline{\top}$.  %\hfill\qedhere
\end{enumerate}}

Finally, if $\set{\alpha,\beta,\gamma}$ satisfies (\ref{nondistributive}) in $L$, 
then $\alpha\vee \beta\not\le \gamma$, and consequently ${\alpha\vee \beta}$ and $x$ are incomparable. Further, referring to $x^-=\gamma$, we observe $({\alpha\vee \beta})\wedge \gamma={({\alpha\vee \beta})\wedge x}$, $\alpha\wedge \gamma=\alpha\wedge x$ and $\beta\wedge \gamma=\beta\wedge x$. Hence %, if we replace $\gamma$ by $x$, then 
$\set{\alpha,\beta,x}$ satisfies (\ref{strictly nondistributive}) in $\overline{L}^\gamma$, i.e.\ $\overline{L}^\gamma$ is strictly nondistributive.
\end{proof}

\begin{remarks}\label{newremarks1} Referring to Lemma~\ref{newnewlemma1.1C} it is worthwhile to note the following:
\begin{enumerate}[label=\textup{(\arabic*)},align=right,leftmargin=0pt,labelwidth=10pt,itemindent=20pt,labelsep=5pt,topsep=5pt,itemsep=3pt]
\item \label{newremarks1.1} If $\gamma\in L\setminus\set{\top}$, then $\top$ and $x$ are incomparable in $\overline{L}^\gamma$.
\item \label{newremarks1.2} The isolated element $x$ in $\overline{L}^\gamma$ is completely join-prime, i.e.\ $x\vartriangleleft x$ holds  in $\overline{L}^\gamma$. 
\end{enumerate}
\end{remarks}

The first two typical examples of strictly nondistributive lattices are the extensions of ${\mathsf{N}}_5$ and ${\mathsf{M}}_3$ by an isolated element in the sense of Lemma~\ref{newnewlemma1.1C}, denoted simply by $\overline{\mathsf{N}}_5$ and $\overline{\mathsf{M}}_3$, whose Hasse diagrams are given in Figure~\ref{fig:N5 and M3}.
 A complete sublattice $S$ of
a complete lattice $L$ is called an \emph{extended pentagon}, respectively an \emph{extended diamond}, if $S$ is isomorphic to  $\overline{\mathsf{N}}_5$, respectively to $\overline{\mathsf{M}}_3$.
Both $\overline{\mathsf{N}}_5$ and $\overline{\mathsf{M}}_3$  have $7$ element. 
\begin{figure}[H] 
         \vskip-5pt
\centering
{\setlength\tabcolsep{7pt}\begin{tabular}{cccc}
         \begin{tikzpicture}[x=7mm,y=7mm,baseline={([yshift=-.5ex]current bounding box.center)},vertex/.style={anchor=base}]
         \draw (1,0)--(0,0.5)--(0,1.5)--(1,2);
         \draw (1,0)--(2,1)--(1,2);
         \draw (0,1.5)--(1,2.);
         \mycircle{1,0};
         \node[right] at (1,0) {$\bot$};
         \mycircle{2,1};
         \node[right] at (2,1) {$\beta$};
         \mycircle{0,1.5};
         \node[left] at (0,1.5) {$\gamma$};
         \mycircle{0,0.5};
         \node[left] at (0,0.5) {$\alpha$};
         \mycircle{1,2};
         \node[right] at (1,2) {$\top$};
         \node[right] at (1,3) {{\color{white}$\overline{\top}$}}; 
        \node[above] at (2.3,-0.35)  {${\mathsf{N}_5}$}; 
         \end{tikzpicture}
&
         \begin{tikzpicture}[x=7mm,y=7mm,baseline={([yshift=-.5ex]current bounding box.center)},vertex/.style={anchor=base}]
         \draw (1,0)--(0,0.5)--(0,1.5)--(1,2)--(1,3);
         \draw (1,0)--(2,1)--(1,2);
         \draw (0,1.5)--(0,2.5)--(1,3);
         \draw (0,1.5)--(1,2.);
         \mycircle{0,2.5};
         \node[left] at (0,2.5) {$x$};
         \mycircle{1,0};
         \node[right] at (1,0) {$\bot$};
         \mycircle{2,1};
         \node[right] at (2,1) {$\beta$};
         \mycircle{0,1.5};
         \node[left] at (0,1.5) {$\gamma$};
         \mycircle{0,0.5};
         \node[left] at (0,0.5) {$\alpha$};
         \mycircle{1,2};
         \node[right] at (1,2) {$\top$};
         \mycircle{1,3};
         \node[right] at (1,3) {$\overline{\top}$};
         \node[above] at (2.3,-0.35)   {$\overline{\mathsf{N}}_5$}; 
         \end{tikzpicture}
&
         \begin{tikzpicture}[x=7mm,y=7mm,baseline={([yshift=-.5ex]current bounding box.center)},vertex/.style={anchor=base}]
         \draw (1,0)--(2,1)--(1,2);
         \draw (1,0)--(0,1)--(1,2);
         \draw (1,0)--(1,1)--(1,2);
         \mycircle{1,0};
         \node[left] at (1,0) {$\bot$};
          \mycircle{0,1};
         \node[left] at (0,1) {$\alpha$};
         \mycircle{1,1};
         \node[left] at (1,1) {$\gamma$};
         \mycircle{2,1};
         \node[right] at (2,1) {$\beta$};
         \mycircle{1,2};
         \node[left] at (1,2) {$\top$};
         \node[right] at (1,3) {{\color{white}$\overline{\top}$}}; 
          \node[above] at (-0.2,-0.35) {${\mathsf{M}_3}$}; 
         \end{tikzpicture}
&
         \begin{tikzpicture}[x=7mm,y=7mm,baseline={([yshift=-.5ex]current bounding box.center)},vertex/.style={anchor=base}]
         \draw (1,0)--(2,1)--(1,2)--(1,3);
         \draw (1,0)--(0,1)--(1,2);
         \draw (1,0)--(1,1)--(1,2);
         \draw (1,1)--(2,2.)--(1,3);
         \draw[white,line width=4pt] (2,1.)--(1,2.);
         \draw (2,1)--(1,2.);
         \mycircle{2,2.};
         \node[right] at (2,2.) {$x$};
         \mycircle{1,0};
         \node[left] at (1,0) {$\bot$};
          \mycircle{0,1};
         \node[left] at (0,1) {$\alpha$};
         \mycircle{1,1};
         \node[left] at (1,1) {$\gamma$};
         \mycircle{2,1};
         \node[right] at (2,1) {$\beta$};
         \mycircle{1,2};
         \node[left] at (1,2) {$\top$};
         \mycircle{1,3};
         \node[left] at (1,3) {$\overline{\top}$}; 
        \node[above] at (-0.2,-0.35) {$\overline{\mathsf{M}}_3$}; 
         \end{tikzpicture}
         \end{tabular}}
\vskip-5pt                 \caption{The lattices $\mathsf{N}_5$ and $\mathsf{M}_3$ and their extensions}
\label{fig:N5 and M3}
         \vskip-10pt
        \end{figure}

We are also interested in two strictly nondistributive lattices with $8$ and $9$ elements:  the extensions of ${\mathsf{L}}_6$ and ${\mathsf{L}}_7$ by an isolated element  in the sense of Lemma~\ref{newnewlemma1.1C}, are  simply denoted by $\overline{\mathsf{L}}_6$ and $\overline{\mathsf{L}}_7$, whose Hasse diagrams are given in Figure~\ref{fig:L6 and L7}.
\begin{figure}[H] 
         \vskip-5pt
\centering
{\setlength\tabcolsep{5pt}\begin{tabular}{cccc}
 \begin{tikzpicture}[x=7mm,y=7mm,baseline={([yshift=-.5ex]current bounding box.center)},vertex/.style={anchor=base}]
         \draw (1,0)--(2,1)--(1,2.25);
         \draw (1,0)--(0,0.75)--(0,1.5)--(1,2.25);
         \draw (1,0)--(1,0.75);
         \draw (1,0.75)--(0,1.5);
         \mycircle{1,0};
         \node[left] at (1,0) {$\bot$};
         \mycircle{0,0.75};
         \node[left] at (0,0.75) {$\alpha$};
         \mycircle{1,0.75};
         \node[above] at (1,0.75) {$\gamma$};
         \mycircle{2,1};
         \node[right] at (2,1) {$\beta$};
         \mycircle{0,1.5};
         \node[above left] at (0.55,1.5) {$\alpha \vee \gamma$};
         \mycircle{1,2.25};
         \node[left] at (1,2.25) {$\top$};
         \node[right] at (1,3) {\color{white}$\overline{\top}$}; 
         \node[above] at (2.3,-0.35)  {${\mathsf{L}}_6$}; 
         \end{tikzpicture} 
& 
         \begin{tikzpicture}[x=7mm,y=7mm,baseline={([yshift=-.5ex]current bounding box.center)},vertex/.style={anchor=base}]
         \draw (1,0)--(2,1)--(1,2.25)--(1,3);
         \draw (1,0)--(0,0.75)--(0,1.5)--(1,2.25);
         \draw (1,0)--(1,0.75);
         \draw (1,0.75)--(0,1.5);
         \draw (1,0.75)--(2,2.25)--(1,3);
         \draw[white,line width=4pt] (2,1)--(1,2.25);
         \draw (2,1.)--(1,2.25);
         \mycircle{2,2.25};
         \node[right] at (2,2.25) {$x$};
         \mycircle{1,0};
         \node[left] at (1,0) {$\bot$};
         \mycircle{0,0.75};
         \node[left] at (0,0.75) {$\alpha$};
         \mycircle{1,0.75};
         \node[above] at (1,0.75) {$\gamma$};
         \mycircle{2,1};
         \node[right] at (2,1) {$\beta$};
         \mycircle{0,1.5};
         \node[above left] at (0.55,1.5) {$\alpha \vee \gamma$};
         \mycircle{1,2.25};
         \node[left] at (1,2.25) {$\top$};
         \mycircle{1,3};
         \node[left] at (1,3) {$\overline{\top}$}; 
         \node[above] at (2.3,-0.35)   {$\overline{\mathsf{L}}_6$}; 
         \end{tikzpicture}
&
         \begin{tikzpicture}[x=7mm,y=7mm,baseline={([yshift=-.5ex]current bounding box.center)},vertex/.style={anchor=base}]
         \draw (1,0)--(2,0.75)--(2,1.5)--(1,2.25);
         \draw (1,0)--(0,0.75)--(0,1.5)--(1,2.25);
         \draw (1,0)--(1,0.75)--(0,1.5);
         \draw (1,0.75)--(2,1.5);
         \draw (2,1.5)--(1,2.25);
         \mycircle{1,0};
         \node[left] at (1,0) {$\bot$};
         \mycircle{0,0.75};
         \node[left] at (0,0.75) {$\alpha$};
         \mycircle{1,0.75};
         \node[above] at (1,0.75) {$\gamma$};
         \mycircle{2,0.75};
         \node[right] at (2,0.75) {$\beta$};
         \mycircle{0,1.5};
         \node[above left] at (0.55,1.5) {$\alpha \vee \gamma$};
         \mycircle{2,1.5};
         \node[above right] at (1.6,1.5) {$\beta \vee \gamma$};
         \mycircle{1,2.25};
         \node[left] at (1,2.25) {${\top}$};
         \node[right] at (1,3) {\color{white}$\overline{\top}$}; 
          \node[above] at (-0.2,-0.35) {${\mathsf{L}}_7$}; 
         \end{tikzpicture}
&
         \begin{tikzpicture}[x=7mm,y=7mm,baseline={([yshift=-.5ex]current bounding box.center)},vertex/.style={anchor=base}]
         \draw (1,0)--(2,0.75)--(2,1.5)--(1,2.25)--(1,3);
         \draw (1,0)--(0,0.75)--(0,1.5)--(1,2.25);
         \draw (1,0)--(1,0.75)--(0,1.5);
         \draw (1,0.75)--(2,1.5);
         \draw (1,0.75)--(2,2.25)--(1,3);
         \draw[white,line width=4pt] (2,1.5)--(1,2.25);
         \draw (2,1.5)--(1,2.25);
         \mycircle{2,2.25};
         \node[right] at (2,2.25) {$x$};
         \mycircle{1,0};
         \node[left] at (1,0) {$\bot$};
         \mycircle{0,0.75};
         \node[left] at (0,0.75) {$\alpha$};
         \mycircle{1,0.75};
         \node[above] at (1,0.75) {$\gamma$};
         \mycircle{2,0.75};
         \node[right] at (2,0.75) {$\beta$};
         \mycircle{0,1.5};
         \node[above left] at (0.55,1.5) {$\alpha \vee \gamma$};
         \mycircle{2,1.5};
         \node[right] at (2,1.5) {$\beta \vee \gamma$};
         \mycircle{1,2.25};
         \node[left] at (1,2.25) {${\top}$};
         \mycircle{1,3};
         \node[left] at (1,3) {$\overline{\top}$}; 
          \node[above] at (-0.2,-0.35) {$\overline{\mathsf{L}}_7$}; 
         \end{tikzpicture}
         \end{tabular}}
 \vskip-5pt                 \caption{The lattices $\mathsf{L}_6$ and $\mathsf{L}_7$ and their extensions}
          \label{fig:L6 and L7}
         \vskip-10pt
        \end{figure}

        Our next result shows the necessity of the existence of one of these $4$ strictly  nondistributive lattices as complete sublattices in any complete strictly nondistributive lattice.

\begin{proposition} \label{newthm1} If $L$ is a complete and strictly nondistributive lattice, then there exists  a complete sublattice of $L$ being isomorphic to $\overline{\mathsf{N}}_5$, $\overline{\mathsf{M}}_3$, $\overline{\mathsf{L}}_6$ or $\overline{\mathsf{L}}_7$.
\end{proposition}

\begin{proof} Let $L$ be a complete and strictly nondistributive lattice. If $\alpha,\beta,x\in L$ satisfy (\ref{strictly nondistributive}), i.e.\ the relations $x\not\le \alpha\vee \beta$ and $x\wedge (\alpha\vee \beta)\not\le
(x\wedge \alpha)\vee (x\wedge \beta)$ hold, then $L$ has at least $5$ distinct elements $\bot$, $\alpha$, $\beta$, $\alpha\vee\beta$ and $x$ satisfying the properties
 \stepcounter{num}
\begin{equation}\label{newnewequation1.3}\alpha\vee\beta\not\le x, \quad x\wedge(\alpha\vee\beta)\not\le \alpha \quad \text{and}\quad  x\wedge(\alpha\vee\beta)\not\le \beta.
\end{equation}
Since $x\not\le \alpha\vee \beta$, (\ref{newnewequation1.3}) implies that $\alpha\vee \beta$ 
and $x$ are incomparable, and consequently we have $x\vee(\alpha\vee\beta)\notin \set{\bot,\alpha,\beta,\alpha\vee\beta,x}$.
Referring again to (\ref{newnewequation1.3}), the relation ${x\wedge(\alpha\vee\beta)}\notin\set{\bot,\alpha,\beta,\alpha\vee\beta,x,x\vee(\alpha\vee\beta)}$ follows. Now  we put $a:=\alpha\vee \beta$ and $\gamma:=x\wedge a$ and infer that 
\[S=\set{\bot,\alpha,\beta,a, \gamma,\alpha\vee \gamma,\beta\vee \gamma,x ,x\vee\alpha,x\vee\beta,x\vee a}\]
is a complete sublattice of $L$ with at least $7$ distinct elements. Moreover (\ref{newnewequation1.3}) 
 implies $\gamma\not\le \alpha$ and  $\gamma\not\le \beta$  and a fortiori $x\not\le \alpha$ and $x\not\le \beta$. Hence $\alpha$ and $\beta$ are atoms in $S$  and $\gamma\neq \bot$.The binary meet operation with $\gamma$ in $S$ restricted to $\alpha$ and $\beta$ has the form:
 \[\alpha\wedge \gamma=\begin{cases} \bot,& \alpha\not\le \gamma,\\ \alpha,& \alpha\le \gamma,\end{cases} \quad \text{and}\quad
\beta\wedge \gamma=\begin{cases} \bot,& \beta\not\le \gamma,\\ \beta,& \beta\le\gamma.\end{cases}
\]
The size of $S$ depends on the values of $\alpha\vee \gamma$, $\beta\vee \gamma$, $x\vee\alpha$ and $x\vee\beta$. 
In order to show that there exist complete sublattices of $S$ being isomorphic to $\overline{\mathsf{N}}_5$, $\overline{\mathsf{M}}_3$, $\overline{\mathsf{L}}_6$ or $\overline{\mathsf{L}}_7$ we distinguish the following cases:
\vskip 2pt

\noindent
{\sf Case 1}. $\alpha\le\gamma$ or $\beta \le \gamma$. Let us assume $\alpha\le \gamma$ --- i.e.\ $\alpha\vee \gamma=\gamma$. Then $\beta\not\le \gamma$ and $\beta\vee \gamma=a$. Since $\gamma\le x$, $x\vee\beta=x\vee a$ and $x\vee \alpha=x$ follow. Hence $S$ is isomorphic to the extended pentagon $\overline{\mathsf{N}}_5$. In particular $x^-=\gamma$ and $x^+=x\vee a$ hold. The case $\beta\le \gamma$ can be treated analogously.
\vskip 2pt

\noindent
{\sf Case 2}. $\alpha\not\le \gamma$ and $\beta\not\le \gamma$.  Then $\gamma$ is the third atom in $S$.
\vskip 2pt
\noindent
{\sf Case 2.1}. $\alpha\not\le \gamma$, $\beta\not\le \gamma$  and $\alpha \vee \gamma=a=\beta\vee \gamma$. Then $\alpha\not\le x$ and $\beta\not\le x$ follow. Hence $S$ is isomorphic to the extended diamond $\overline{\mathsf{M}}_3$ with three atoms. In particular $x^-=\gamma$ and $x^+=x\vee a$.
\vskip 2pt
\noindent
{\sf Case 2.2}. $\alpha\not\le \gamma$, $\beta\not\le \gamma$, $\alpha\vee \gamma=a$ and $\beta \vee \gamma\neq a$. Then again $\alpha\not\le x$, but we cannot control the value of $x\vee \beta$. Therefore we distinguish two further cases:\\
{\sf Case 2.2.1}. $\beta \le x$. Then $x\vee \alpha=x\vee a$, and $S$ satisfies the following Hasse diagram:
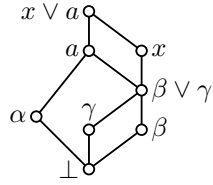
\begin{center} 
\begin{tikzpicture}[x=7mm,y=7mm,baseline={([yshift=-.5ex]current bounding box.center)},vertex/.style={anchor=base}]
         \draw (1,0)--(2,0.75)--(2,1.5)--(2,2.25)--(1,3);
       \draw (1,0)--(1,0.75)--(2,1.5)--(1,2.25)--(1,3);
         \draw (1,0)--(0,1)--(1,2.25);
         \mycircle{2,2.25};
         \node[right] at (2,2.25) {$x$};
         \mycircle{1,0};
         \node[left] at (1,0) {$\bot$};
         \mycircle{0,1};
         \node[left] at (0,1) {$\alpha$};
         \mycircle{1,0.75};
         \node[above] at (1,0.75) {$\gamma$};
         \mycircle{2,0.75};
         \node[right] at (2,0.75) {$\beta$};
         \mycircle{2,1.5};
         \node[right] at (2,1.5) {$\beta \vee \gamma$};
         \mycircle{1,2.25};
         \node[left] at (1,2.25) {$a$};
         \mycircle{1,3};
         \node[left] at (1,3) {$x \vee a$}; 
         \end{tikzpicture}
\end{center}
Obviously $\set{\bot, \alpha,\gamma, \beta\vee \gamma, a, x , x\vee a}$ is a complete sublattice of $S$ being isomorphic to $\overline{\mathsf{N}}_5$.\\
{\sf Case 2.2.2}. $\beta\not\le x$. Then the Hasse diagram of $S$ has the form:
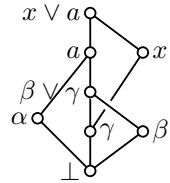
\begin{center} 
\begin{tikzpicture}[x=7mm,y=7mm,baseline={([yshift=-.5ex]current bounding box.center)},vertex/.style={anchor=base}]
         \draw (1,0)--(2,0.75);
         \draw (1,0)--(1,0.75)--(1,1.5)--(1,2.25)--(1,3);
         \draw (1,0)--(0,1)--(1,2.25);
         \draw (1,0.75)--(2,2.25)--(1,3);
         \draw[white,line width=4pt]  (2,0.75)--(1,1.5);
         \draw (2,0.75)--(1,1.5);
         \mycircle{2,2.25};
         \node[right] at (2,2.25) {$x$};
         \mycircle{1,0};
         \node[left] at (1,0) {$\bot$};
         \mycircle{0,1};
         \node[left] at (0,1) {$\alpha$};
         \mycircle{1,0.75};
         \node[right] at (1,0.75) {$\gamma$};
         \mycircle{2,0.75};
         \node[right] at (2,0.75) {$\beta$};
         \mycircle{1,1.5};
         \node[left] at (1,1.5) {$\beta \vee \gamma$};
         \mycircle{1,2.25};
         \node[left] at (1,2.25) {$a$};
         \mycircle{1,3};
         \node[left] at (1,3) {$x \vee a$}; 
         \end{tikzpicture}
\end{center}
and consequently $S$ is isomorphic to $\overline{\mathsf{L}}_6$ (cf.\ Figure~\ref{fig:L6 and L7}). In particular  $x^-=\gamma$ and $x^+=x\vee a$ hold.\\
If we interchange the role of $\alpha$ and $\beta$ in the previous consideration, then the case $\alpha\not\le \gamma$, $\beta\not\le \gamma$, $\alpha\vee \gamma\neq a$ and $\beta \vee \gamma=a$ can be treated analogously.
\vskip 2pt
\noindent
{\sf Case 2.3}. $\alpha\not\le \gamma$, $\beta\not\le \gamma$, $\alpha\vee \gamma\neq a$ and $\beta \vee \gamma\neq a$. Then $\beta\not\le \alpha \vee \gamma$ and $\alpha\not\le \beta\vee \gamma$.\\
{\sf Case 2.3.1}. Let us further assume $\alpha\not\le x$ and $\beta\not\le x$ --- i.e.\ $x\vee\alpha\neq x$ and $x\vee \beta\neq x$. If $x\vee \alpha=x\vee a=x\vee \beta$, then $S$ is isomorphic to  $\overline{\mathsf{L}}_7$ (cf. Figure 2) with $x^-=\gamma$ and $x^+=x\vee a$. On the other hand, if $x\vee \alpha \neq x\vee a$ or $x\vee \beta \neq x\vee a$, then there exists a complete sublattice of $S$ being isomorphic to $\overline{\mathsf{N}}_5$. In fact, let us assume $x\vee \alpha\neq x\vee a$. Then the complete sublattice $\set{\bot, \alpha,\beta, \alpha\vee \gamma, a , x\vee \alpha,x\vee a}$ is isomorphic to $\overline{\mathsf{N}}_5$. The case $x\vee \beta\neq x\vee a$ can be treated analogously.\\
{\sf Case 2.3.2}. $\alpha\le x$ and $\beta\not\le x$. Then $\beta \vee x=x\vee a$ and $\set{\bot,\beta, \alpha\vee \gamma, a,x,x\vee a}$ is a complete sublattice of $S$ being isomorphic to $\overline{\mathsf{N}}_5$. The case $\alpha\not\le x$ and $\beta\le x$ can be treated analogously.
\end{proof}

\begin{comments} \label{newnewcomment2} (a) We maintain the notation from the proof of Proposition~\ref{newthm1} and observe that $P=\set{\bot,\alpha,\beta,\gamma,\alpha\vee \gamma, \beta\vee \gamma,a}$with the atoms $\alpha$ and $\beta$ satisfying $\alpha\vee \beta =a$ is a complete sublattice of $S$. % in the sense of $\cat{Sup}$.
 Since in the Cases~1, 2.1, 2.2.2 and in the first part of the Case 2.3.1 the properties $x^-=\gamma$ and $x^+=x\vee a$ hold,  Lemma~\ref{newnewlemma1.1C} implies that $S$ coincides with the extension $\overline{P}^\gamma=P\cup\set{x,x\vee a}$ by an isolated element $x$. This construction does not apply to the remaining cases in the proof of Proposition~\ref{newthm1}.\\[1mm]
(b) The converse of Proposition~\ref{newthm1} cannot hold, since in $\cat{Sup}$ every complete lattice $L$ can be embedded into a completely distributive lattice --- e.g.\
\[L\xrightarrow{\,\varphi\,} P(L)^{op},\quad \varphi(\alpha)=\tbigcup\set{A\in P(L)\mid \alpha\le \tbigwedge A}, \quad \alpha\in L,\]
where $P(L)^{op}$ is the dual lattice of the power set of $L$. 
\end{comments}

\begin{corollary}\label{newcor1} Any strictly nondistributive lattice must have at least $7$ elements, and if it has exactly $7$ elements, then it must be either an extended pentagon or an extended diamond.
\end{corollary}

\section{Construction of unitally  nondistributive quantales}\label{sec:2}

As we announced in the introduction, we say that a unital quantale $\alg{Q}=(\alg{Q},\ast,e)$ is 
\emph{unitally nondistributive} if it satisfies the
following properties:
\begin{enumerate}[label=\textup{(\arabic*)},align=right,leftmargin=0pt,labelwidth=10pt,itemindent=20pt,labelsep=5pt,topsep=5pt,itemsep=3pt]
\item $\tbigvee\set{\alpha\in \alg{Q}\mid \alpha
\vartriangleleft e}=e$ (the unit $e$ is
\donotbreakdash{$\vartriangleleft$}approximable).
\item There exists a subset $A$ of $\alg{Q}$ such that
$e\wedge\bigl(\tbigvee A\bigr)\not\le \tbigvee_{\alpha\in A}
(e\wedge \alpha)$.
\end{enumerate}

\begin{proposition}\label{newprop3}
If $\alg{Q}$ is a finite unitally nondistributive quantale with unit $e$ then
there exist $\alpha,\beta\in \alg{Q}$ such that \stepcounter{num}
\begin{equation}\label{n.1C}e\wedge (\alpha\vee \beta)\not\le
(e\wedge \alpha)\vee (e\wedge \beta)\quad \text{and}\quad e\not\le \alpha\vee \beta.
\end{equation}
Hence the underlying lattice of $\alg{Q}$ is strictly nondistributive.
\end{proposition}

\begin{proof} Let $A=\set{\alpha_1,\ldots,\alpha_n}$ be a subset of  $\alg{Q}$ satisfying (2). Then the cardinality of $A$ is at least $2$, and we define $\beta_i:= {\alpha_i\vee\cdots\vee\alpha_n}$ for each $i\in
\set{1,\ldots,n}$. If we assume $e\wedge \beta_i\le (e\wedge \alpha_i)\vee
(e\wedge \beta_{i+1})$ for each $i\in \set{1,\ldots,n-1}$, then
we would have
\begin{align*}
e\wedge\bigl(\tbigvee A\bigr)&=e\wedge\beta_1\le
(e\wedge \alpha_1)\vee (e\wedge
\beta_{2})\le \cdots\le (e\wedge\alpha_1)\vee\cdots\vee(e\wedge\alpha_n),
\end{align*}
which is a contradiction to (2). Hence
there exists $i\in \set{1,\ldots,n-1}$ such that the relation $e\wedge
\beta_i\not\le (e\wedge \alpha_i)\vee (e\wedge \beta_{i+1})$ holds.
If we now put $\alpha:=\alpha_{i}$ and $\beta:=\beta_{i+1}$,
then $\alpha\vee \beta=\beta_i$, and the first part of (\ref{n.1C}) is verified.
 Further, since the unit $e$ is \donotbreakdash{$\vartriangleleft$}approximable, the first part of (\ref{n.1C}) implies the second part of (\ref{n.1C}). 
\end{proof}

The same arguments can be used for unitally nondistributive quantales such that the underlying
lattice is meet continuous, i.e.\ binary meets distribute over
directed suprema. Hence Proposition~\ref{newprop3} holds also for  this class of quantales. 

In view of Proposition~\ref{newprop3} and Corollary~\ref{newcor1} we conclude that every unitally nondistributive and meet continuous quantale contains at least $7$ elements. This observation is a motivation to give a classification of  all unitally nondistributive  quantales on the set with $7$ elements (see Subsection~\ref{subsec:3.1} infra).

%%%%%%%%%%%%%%%%%%%%%%%%%%%%%%%%%%%%%%%%%%%%%%%%%%%%%%%%%%%%%%%%%%%%%%%%%%%%%%%%%%%%%%%%%%%%%%%%%%%%%%%%%%%%%%%%%%%%%%%%%%%%%%%%%%%%%%%%%%%%%%%%%%%%%%%%%%%%%%

%%%%%%%%%%%%%%%%%%%%%%%%%%%%%%%%%%%%%%%%%%%%%%%%%%%%%%%%%%%%%%%%%%%%%%%%%%%%%%%%%%%%%%%%%%%%%%%%%%%%%%%%%%%%%%%%%%%%%%%%%%%%%%%%%%%%%%%%%%%%%%%%%%%%%%%%%%%%%%%%%%%%%%%%%%%%%%%%%%%%%%%%%%%%%%%%%%%%%%%%%%%%%%%%%%%%%%%%%%%%%%%%%%%%%%%%%%%%
%%%%%%%%%%%%%%%%%%%%%%%%%%%%%%%%%%%%%%%%%%%%%%%%%%%%%%%%%%%%%%%%%%%%%%%%%%%%%%%%%%%%%%%%%%%%%%%%%%%%%%%%%%%%%%%%%%%%%%%%%%%%%%%%%%%%%%%%%%%%%%%%%%%%%%%%%%%%%%%%%%
Based on the observation and results of Section~\ref{sec:1} we will
present a procedure describing an extension of a quantale to a unitally
nondistributive quantale. This construction is based on the extension of the underlying lattice by an isolated element ({cf.\ Lemma}~\ref{newnewlemma1.1C}) which in this case will be the unit of the new quantale. In contrast with the previous situation, where \emph{any} lattice could be extended in this way, we need some additional condition on the quantale in order to have the required extension. We first note the following:

\begin{lemma}\label{newlem4}
Let $L$ be a complete lattice, $\gamma\in L\setminus\set{\top}$ and let $\alg{Q}$ be a unital quantale on the extension $\overline{L}^\gamma=L\cup\set{e,\overline{\top}}$ of $L$ in the sense of Lemma~\ref{newnewlemma1.1C} such that the isolated element $e$ is the unit of $\alg{Q}$. Then\textup:
\begin{enumerate}[label=\textup{(\arabic*)},leftmargin=20pt,labelwidth=10pt,itemindent=0pt,labelsep=5pt,topsep=5pt,itemsep=3pt]
\item\label{newlem4(1)} $(\gamma\ast\alpha)\vee( \alpha\ast\gamma)\le\alpha$ for all $\alpha\in L$,
\item\label{newlem4(2)} $\overline{\top}\ast\alpha=(\beta\ast \alpha)\vee\alpha$ and $\alpha\ast\overline{\top}=( \alpha\ast\beta)\vee\alpha$ for all $\beta\not\le \gamma$ and $\alpha\in L$,
\item\label{newlem4(3)} $L$ is a subquantale of $\alg{Q}$ if and only if $\top\ast\top\le \top$.
\item\label{newlem4(4)} If $\bigwedge\set{\beta \vee \gamma\mid \beta\in L, \ \beta\not\le \gamma}\le \gamma$, then $\gamma$ is two-sided in $\overline{L}^\gamma$. %i.e.\ $(\overline{\top}\ast\gamma)\vee(\gamma\ast\overline{\top})\le\gamma$.
\item\label{newlem4(5)} If $L$ is a subquantale of $\alg{Q}$, then the quantale multiplication in $\alg{Q}$ with unit $e$ is uniquely determined by its values in $L$. 
\end{enumerate}
\end{lemma}

\begin{proof} The first assertion follows immediately since $\gamma\le e$ and $e$ is the unit of $\alg{Q}$. Since $\gamma\neq \top$, the set $\set{\beta\in L\mid\beta\not\le \gamma}$ is nonempty, and consequently \ref{newlem4(2)} follows from $\overline{\top}=\beta\vee e$ for all $\beta\not\le \gamma$ (cf.\ \ref{(viii)}). The third one is obvious. The assertion \ref{newlem4(4)} follows from \ref{newlem4(1)} and \ref{newlem4(2)}.
 Finally \ref{newlem4(5)} is an immediate corollary of \ref{newlem4(2)}.
\end{proof}

Now we have the following:

\begin{theorem}\label{newthm2} Let $\alg{Q}$ be a quantale, $\gamma\in \alg{Q}\setminus\{\top\}$ and let $\overline{\alg{Q}}^\gamma=\alg{Q}\cup\set{\overline{\top},e}$ be the extension of the underlying lattice  by an isolated element $e$ in the sense of Lemma~\ref{newnewlemma1.1C}. Then there exists a unique quantale structure on $\overline{\alg{Q}}^\gamma$ with unit $e$ and subquantale $\alg{Q}$ if and only if $\alg{Q}$ satisfies  the following conditions for all $\alpha\in \alg{Q}$ and $\beta\in \alg{Q}$ with $\beta\not\le \gamma$\textup:
\stepcounter{num}
\begin{eqnarray}\label{property(a)}
&(\gamma\ast\alpha)\vee( \alpha\ast\gamma)\le\alpha,\\
\stepcounter{num} \label{property(aa)}
&\top\ast \alpha\le (\beta \ast \alpha)\vee \alpha\quad \text{and} \quad \alpha\ast \top\le (\alpha\ast \beta)\vee \alpha.\end{eqnarray}
% Then $\alg{Q}$ is a subquantale of {\color{blue}the} unital quantale
%$\overline{\alg{Q}}^\gamma$ {\color{blue}with unit $e$ such that $e^+=\overline{\top}$, $e^-=\gamma$ and $e$}
%is \donotbreakdash{$\vartriangleleft$}approximable.\\[1pt]
Moreover, if there exist two elements $\alpha,\beta\in \alg{Q}$
satisfying the property
\stepcounter{num}
\begin{equation}\label{property(b)}\gamma\wedge (\alpha\vee \beta)\not\le
(\gamma\wedge \alpha)\vee (\gamma\wedge \beta),\end{equation}
then $\overline{\alg{Q}}^\gamma$ is a unitally
nondistributive quantale.
\end{theorem}

\begin{proof} The necessity follows from Lemma~\ref{newlem4}\,\ref{newlem4(1)} and \ref{newlem4(2)}, and the uniqueness of the quantale structure with unit $e$ and $\alg{Q}$ as subquantale follows from  Lemma~\ref{newlem4}\,\ref{newlem4(5)}. In order to establish the existence of the quantale structure on $\overline{\alg{Q}}^\gamma$ with unit $e$ and subquantale $\alg{Q}$ we extend the quantale multiplication on $\alg{Q}$ to the additional elements $e$ and $\overline{\top}$ as follows. First we require that $e$ is the unit  in  $\overline{\alg{Q}}^\gamma$, and secondly the multiplication involving $\overline{\top}$ is determined by:
\stepcounter{num}
\begin{equation}\label{newmultiplication}\overline{\top}\ast \overline{\top}=\overline{\top},\quad \overline{\top}\ast\alpha=(\top\ast\alpha)\vee\alpha,\quad \alpha\ast \overline{\top}=(\alpha\ast
\top)\vee\alpha,\qquad\alpha\in \alg{Q}.\end{equation}
Since $\gamma=e^-$, the property (\ref{property(a)}) implies that this extension of the quantale multiplication is isotone in $\overline{\alg{Q}}^\gamma$. Further, the property (\ref{property(aa)}) guarantees that this extension  preserves arbitrary joins in each variable separately, i.e.\ it is a quantale multiplication in $\overline{\alg{Q}}^\gamma$ with unit $e$. 
In fact, it follows from (\ref{newmultiplication}) that the left and right multiplications with $\overline{\top}$ are obviously  join-preserving. The interesting situation occurs when we consider $\alpha\in \alg{Q}$ and a subset $A$ of $\alg{Q}\cup\set{e}$ such that $\tbigvee A=\overline{\top}$.  Then we have necessarily that $e\in A$ and there exists $\beta\in \alg{Q}$ with $\beta\not\le \gamma$ and $\beta\in A$. Now we apply (\ref{newmultiplication}) and (\ref{property(aa)}) and obtain:
\[(\tbigvee A)\ast\alpha= (\top\ast\alpha ) \vee \alpha \le (\beta\ast\alpha )\vee \alpha \le \tbigvee\set{\varepsilon\ast\alpha\mid \varepsilon\in A}\le  (\top\ast\alpha ) \vee \alpha .\] 
Hence $(\tbigvee A)\ast\alpha= \tbigvee\set{\varepsilon\ast\alpha\mid \varepsilon\in A}$, and analogously we verify $\alpha\ast(\tbigvee A)= \tbigvee\set{\alpha\ast\varepsilon\mid \varepsilon\in A}$. 
Finally, by Lemma~\ref{newlem4}\,\ref{newlem4(3)}, $\alg{Q}$ is a subquantale of  $\overline{\alg{Q}}^\gamma$.\\
 
On the other hand, if there exist 
$\alpha,\beta\in \alg{Q}$ satisfying (\ref{property(b)}), then Lemma~\ref{newnewlemma1.1C} implies that the underlying lattice of $\overline{\alg{Q}}^\gamma$ is strictly nondistributive with isolated unit $e$, and so \ref{newlem4(2)} holds. Since $e$ is completely join-prime (cf.\  Remark~\ref{newremarks1}\,(2)) and in particular \donotbreakdash{$\vartriangleleft$}approximable, we conclude that
$\overline{\alg{Q}}^\gamma$ is a unitally nondistributive quantale.
\end{proof}

\begin{remarks} \label{newremarkCC} (1) If $\alg{Q}=(\alg{Q},\ast,\delta)$ is a unital quantale with unit $\delta$, then $\alg{Q}$ satisfies (\ref{property(a)}) if and only if $\gamma\le\delta$. \\[2pt]
(2) If $\beta\vee\gamma=\top$ for each $\beta\in \alg{Q}$ with $\beta\not\le \gamma$ then (\ref{property(a)}) trivially implies (\ref{property(aa)}). \\[2pt]
(3) In general (\ref{property(a)}) does not imply (\ref{property(aa)}).
 Indeed, let $\alg{Q}$ be a unital quantale on $\mathsf{L}_6$  such that $\gamma$ is the unit (referring to \cite{Catalogue} there exist at least $12$ quantales of this type). 
 Then $\alg{Q}$ satisfies (\ref{property(a)}), but we recall that $\alpha\not\le \gamma$ and $\alpha\vee\gamma\not\le \top$ (cf. Figure~\ref{fig:L6 and L7}) and observe that $\top\ast \gamma=\top\not\le (\alpha\ast \gamma)\vee \gamma$. Hence the property (\ref{property(aa)}) does not hold. 
\end{remarks}

\begin{examples}\label{examples1} (1) Every two-sided quantale
satisfies the properties  (\ref{property(a)}) and (\ref{property(aa)}) ---  in particular, the trivial quantale
determined by the condition  $\top\ast \top=\bot$. Consequently, if the underlying lattice of a two-sided quantale $\alg{Q}$ contains elements $\alpha,\beta, \gamma\in \alg{Q}$
satisfying  (\ref{property(b)}),  then $\overline{\alg{Q}}^\gamma$ is a unitally
nondistributive quantale. The simplest examples of this situation are precisely the trivial quantales on the pentagon and on the diamond with $3$ atoms.
\vskip 3pt

\noindent
(2) A further source of unitally nondistributive quantales consists of arbitrary groups with at least  three elements. For this purpose we first recall the construction how unital quantales are induced by arbitrary  groups $G=(G,\cdot,e)$ (cf.\ \cite[Sect.~5]{GHK21b}). First we provide $G$ with the discrete order  and  subsequently we apply the MacNeille completion. This construction leads to a complete lattice $G_{\infty}=G \cup\set{\bot,\top}$ by adding the universal bounds to the discretely ordered set $G$. 
In a second step we extend the group multiplication on $G$ to a quantale
multiplication $\ast$ on $G_{\infty}$ as follows:
\begin{enumerate}[label=\textup{--},leftmargin=15pt,labelwidth=10pt,itemindent=0pt,labelsep=5pt,topsep=5pt,itemsep=3pt]
\item If $\alpha,\beta\in G$, then $\alpha \ast
\beta=\alpha\cdot \beta$.
\item The action of the universal bounds $\bot$ and $\top$
on elements of $G$ is determined by:\\
$\top\ast \top =\top$, $\bot \ast \bot =\bot$, $\top \ast
\alpha=\top=\alpha\ast \top$ and $\bot\ast
\alpha=\bot=\alpha\ast \bot$ for all $\alpha\in G$.
\end{enumerate}
The unit of $G_{\infty}$ is the unit $e$ of the group $G$. If we choose the element $\gamma$ as the unit of the quantale $(G_{\infty},\ast,e)$, i.e.\ $\gamma=e$, then  $(G_{\infty},\ast,e)$ satisfies the properties (\ref{property(a)}) and (\ref{property(aa)}), and if $G$ contains at least three different elements, then $\gamma=e$ satisfies also (\ref{property(b)}). 
\pagebreak
Hence, if a group $G$ contains at least three elements,  we conclude from Theorem~\ref{newthm2} that the unital quantale $(G_{\infty},\ast,e)$ can always be extended to unitally nondistributive quantale.  In this sense Theorem~\ref{newthm2} produces an extremely large amount of this type of quantales. The simplest example of this situation is when $G$ contains  precisely three different elements. In this case $(G_{\infty},\ast,e)$ is the quantale induced by the cyclic group on the diamond with $3$ atoms.
\end{examples}

%%%%%%%%%%%%%%%%%%%%%%%%%%%%%%%%%%%%%%%%%%%%%%%%%%%%%%%%%%%%%%%%%%%%%%%%%%%%%%%%%%%%%%%%%%%%%%%%%%%%%%%%%%%%%%%%%%%%%%%%%%%%%%%%%%%%%%%%%%%%%%%%%%%%%%%%%%%%%%%%%%%%%%%%%%%%%%%%%%%%%%%%%%%%%%%%%%%%%%%%%%%%%%%%%%%%%%%%%%%%%%%%%%%%%%%%%%%%%%%%%%%
%%%%%%%%%%%%%%%%%%%%%%%%%%%%%%%%%%%%%%%%%%%%%%%%%%%%%%%%%%%%%%%%%%%%%%%%%%%%%%%%%%%%%%%%%%%%%%%%%%%%%%%%%%%%%%%%%%%%%%%%%%%%%%%%%%%%%%%%%%%%%%%%%%%%%%%%%%%%%%%%%%

\section{Unitally nondistributive quantales on finite sets}
\label{sec:3}

 Let $\alg{Q}$ be a (finite) unitally nondistributive  quantale with isolated unit $e$. Then  the underlying lattice is strictly nondistributive 
 (cf.\   Proposition~\ref{newprop3}). In particular, there exist elements $\alpha,\beta\in \alg{Q}$ such that $\set{\alpha,\beta,e}$ satisfies (\ref{n.1C}). Then $\alpha$ and $\beta$ are incomparable. We put $a:=\alpha\vee \beta$, $\gamma :=e\wedge(\alpha\vee \beta)$ and conclude from the proof of Proposition~\ref{newthm1} that also $e$ and $a$ are incomparable and the relations $\gamma\not\le \alpha$ and $\gamma\not\le \beta$  hold. In particular $\gamma\neq \bot$. Further, we assume  $e^-= \gamma$ and $e^+=e\vee a$. Then  
 \[S=\set{\bot, \alpha,\beta,a,\gamma, \alpha\vee \gamma, \beta\vee \gamma, e, e\vee a}\]
is a complete sublattice of the underlying lattice of $\alg{Q}$, $e$ is isolated in $S$, and finally, $\alpha$ and $\beta$ are atoms in $S$.

In general $S$ is not a subquantale of $\alg{Q}$, but it is always a strictly nondistributive lattice. In contrast to Section~\ref{sec:1} we now assume a  quantale structure on $S$ such that the element $e$ plays the role of an isolated unit. Since $e\vartriangleleft e$ holds in $S$, every unital quantale on $S$ with $e$ as unit is a unitally nondistributive quantale. The aim of the following considerations is to give a characterization of such quantales on $S$.

\begin{proposition} \label{newnewprop6} Let $\alg{Q}$ be a unital quantale on $S$ with unit $e$ such that $e^-=\gamma$ and $e^+=e\vee a$.  Then the relation $a\ast a\le a$ holds, i.e.\ $P=\set{\bot,\alpha,\beta,\gamma, \alpha\vee \gamma, \beta\vee\gamma, a}$ is a subquantale of $\alg{Q}$. Moreover $P$ satisfies the conditions {\rm(\ref{property(a)})} and 
{\rm(\ref{property(aa)})}.
\end{proposition}

\begin{proof} Since $e^-=\gamma$ and $e^+=e\vee a$, Lemma~\ref{newnewlemma1.1C} (see also Comment~\ref{newnewcomment2}\,(a)) points out that $S$ has the form $S=\overline{P}^\gamma=P\cup\set{e,e\vee a}$. 

Now we assume $a\ast a\not\le a$. Then  $a\ast a\in \set{e,e\vee a}$, this means:
\[e\le  a\ast a= (\alpha \ast \alpha) \vee (\alpha \ast \beta)\vee (\beta \ast \alpha) \vee (\beta \ast \beta).\]
Since $e \vartriangleleft e$, there exists $x,y\in \set{\alpha,\beta}$ satisfying the property $e\le x\ast y$. If $x\ast y=e$, first we notice that
the right quantale multiplication by $x$ is an injective, join-preserving self-map of $\alg{Q}$ and consequently an order automorphism, since $\alg{Q}$ is finite.  In particular there exists $z\in \alg{Q}$ such that $z\ast x=e$, and $x\ast y=e$ implies $z=y$ --- i.e.\ this order automorphism sends $y$ to $e$. Since $y$ is an atom and $e$ is not an atom, this cannot happen.
Hence we have $x\ast y=a \vee e$. Since $\gamma\le e$ and $x$ is an atom, we obtain ${\gamma\ast x\in \set{\bot,x}}$. If $\gamma \ast x= \bot$, then $\gamma=\gamma \ast e\le \gamma \ast x \ast y= \bot$ follows, a contradiction. If $\gamma \ast x=x$, then 
\[ e\le x\ast y =\gamma \ast x \ast y=\gamma \ast (a \vee e)=(\gamma \ast a) \vee \gamma\le(e \ast a) \vee a = a \] 
follows. Hence we obtain again a contradiction and
   the assumption at the beginning is false  --- i.e.\  $a \ast a \le a$ holds.\\
Finally, (\ref{property(a)})  follows from $\gamma\le e$.  With regard to (\ref{property(aa)}) we 
choose $\delta\in P$ with $\delta\not\le \gamma$. Since $e^-=\gamma$ and $e^+=e\vee a$, we conclude $\delta\not\le  e$, which implies $\delta \vee e=e\vee a$. Hence for $\lambda\in P$ we obtain $a\ast\lambda\le (e\vee a)\ast \lambda=(\delta \vee e)\ast \lambda=(\delta\ast \lambda)\vee \lambda$.
Analogously, we verify $\lambda\ast a\le (\lambda\ast \delta)\vee \lambda$. 
\end{proof}

%%%%%%%%%%%%%%%%%%%%%%%%%%%%%%%%%%%%%%%%%%%%%%%%%%%%%%%%%%%%%%%%%%%%%%%%%%%%%%%%%
\pagebreak

 As an immediate corollary of Theorem~{\ref{newthm2}  and Proposition~\ref{newnewprop6} we obtain the following:
 
 \begin{result}\label{result} Every unital quantale on $S$ with isolated  unit $e$ satisfying $e^-=\gamma$ and $e^+=e\vee a$ is an extension of a quantale on $P$ provided with properties $($\ref{property(a)}$)$ and $($\ref{property(aa)}$)$.
 \end{result}

In the following considerations we will treat different versions of $S$ depending on the values of $\alpha\vee \gamma$ and $\beta\vee \gamma$.
%%%%%%%%%%%%%%%%%%%%%%%%%%%%%%%%%%%%%%%%%%%%%%%%%%%%%%%%%%%%%%%%%%%%%%%%%%%%%%%%%%%%%%%%%%%%%%%%%%%%%%%%%%%%%%%%%%%%%%%%%%%%%%%%%%%%%%%%%%%%%%%%%%%%%%%%%%%%%%

\subsection{Unitally nondistributive quantales on a complete lattice with $7$ elements}
\label{subsec:3.1}
As a first step we consider the case  $\alpha\vee \gamma=a=\beta\vee \gamma$. Then  $\alpha\not\le \gamma$ and $\beta\not\le \gamma$  follow, and $\gamma$ is a third atom in $S$.  In this situation the lattice-structure of $S$ is characterized by the following Hasse diagram:
 \stepcounter{num}
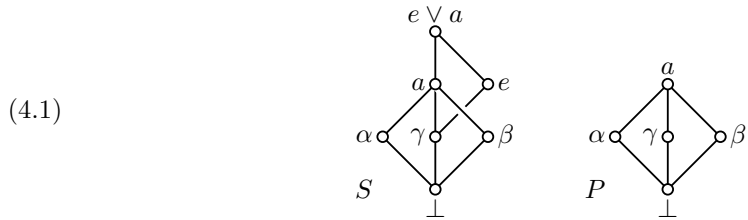
\begin{equation}\label{n3.1C}
         \begin{tikzpicture}[x=7mm,y=7mm,baseline={([yshift=-.5ex]current bounding box.center)},vertex/.style={anchor=base}]
         \draw (1,0)--(2,1)--(1,2)--(1,3);
         \draw (1,0)--(0,1)--(1,2);
         \draw (1,0)--(1,1)--(1,2);
         \draw (1,1)--(2,2.)--(1,3);
         \draw[white,line width=4pt] (2,1.)--(1,2.);
         \draw (2,1)--(1,2.);
         \mycircle{2,2.};
         \node[right] at (2,2.) {$e$};
         \mycircle{1,0};
         \node[below] at (1,0) {$\bot$};
         \mycircle{0,1};
         \node[left] at (0,1) {$\alpha$};
         \mycircle{1,1};
         \node[left] at (1,1) {$\gamma$};
         \mycircle{2,1};
         \node[right] at (2,1) {$\beta$};
         \mycircle{1,2};
         \node[left] at (1,2) {$a$};
         \mycircle{1,3};
         \node[above] at (1,3) {$e \vee a$}; 
         \node[left] at (0,0) {$S$}; 
         \end{tikzpicture}
         \qquad
         \begin{tikzpicture}[x=7mm,y=7mm,baseline={([yshift=-.5ex]current bounding box.center)},vertex/.style={anchor=base}]
         \draw (1,0)--(2,1)--(1,2);
         \draw (1,0)--(0,1)--(1,2);
         \draw (1,0)--(1,1)--(1,2);
         \mycircle{1,0};
         \node[below] at (1,0) {$\bot$};
         \mycircle{0,1};
         \node[left] at (0,1) {$\alpha$};
         \mycircle{1,1};
         \node[left] at (1,1) {$\gamma$};
         \mycircle{2,1};
         \node[right] at (2,1) {$\beta$};
         \mycircle{1,2};
         \node[above] at (1,2) {$a$};
         \node[above,white] at (1,3) {$e \vee a$}; 
         \node[left] at (0,0) {$P$}; 
         \end{tikzpicture}
\end{equation}
%\vskip-5pt 
\noindent 
and $P$ is the diamond with $3$ atoms. 
Since $\alpha\vee\gamma=\beta\vee\gamma=a\vee \gamma=a$, we conclude from Remarks~\ref{newremarkCC}\,(2) and Result~\ref{result} that every quantale on $S$ with unit $e$ is an extension of a quantale  by an isolated unit on  the diamond ${\mathsf{M}}_3$ satisfying (\ref{property(a)}). 

In the second step we consider the case $\alpha\le \gamma$, i.e.\ $\alpha\vee \gamma=\gamma$. Since $\alpha\vee \beta=a$, the relation ${\beta\vee \gamma}=a$ follows, and consequently we have $\beta\not\le \gamma$.  Hence the lattice-structure of $S$ is characterized by the following Hasse diagram:
 \stepcounter{num}
 %   \vskip-5pt 
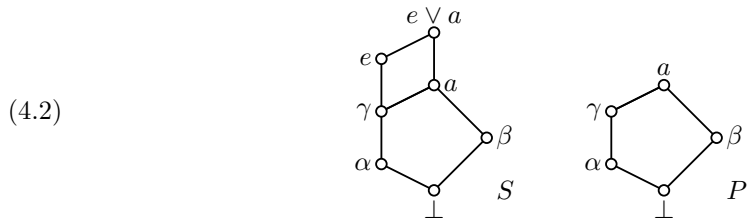
\begin{equation}\label{S71}
         \begin{tikzpicture}[x=7mm,y=7mm,baseline={([yshift=-.5ex]current bounding box.center)},vertex/.style={anchor=base}]
         \draw (1,0)--(0,0.5)--(0,1.5)--(1,2)--(1,3);
         \draw (1,0)--(2,1)--(1,2);
         \draw (0,1.5)--(0,2.5)--(1,3);
         \draw (0,1.5)--(1,2.);
         \mycircle{0,2.5};
         \node[left] at (0,2.5) {$e$};
         \mycircle{1,0};
         \node[below] at (1,0) {$\bot$};
         \mycircle{2,1};
         \node[right] at (2,1) {$\beta$};
         \mycircle{0,1.5};
         \node[left] at (0,1.5) {$\gamma$};
         \mycircle{0,0.5};
         \node[left] at (0,0.5) {$\alpha$};
         \mycircle{1,2};
         \node[right] at (1,2) {$a$};
         \mycircle{1,3};
         \node[above] at (1,3) {$e\vee a$};
         \node[right] at (2,0) {$S$}; 
         \end{tikzpicture}
         \qquad
         \begin{tikzpicture}[x=7mm,y=7mm,baseline={([yshift=-.5ex]current bounding box.center)},vertex/.style={anchor=base}]
         \draw (1,0)--(0,0.5)--(0,1.5)--(1,2);
         \draw (1,0)--(2,1)--(1,2);
         \draw (0,1.5)--(1,2.);
         \mycircle{1,0};
         \node[below] at (1,0) {$\bot$};
         \mycircle{2,1};
         \node[right] at (2,1) {$\beta$};
         \mycircle{0,1.5};
         \node[left] at (0,1.5) {$\gamma$};
         \mycircle{0,0.5};
         \node[left] at (0,0.5) {$\alpha$};
         \mycircle{1,2};
         \node[above] at (1,2) {$a$};
         \node[above,white] at (1,3) {$e\vee a$};
         \node[right] at (2,0) {$P$}; 
         \end{tikzpicture}
\end{equation}
%\vskip-5pt 
\noindent and $P$ is  the pentagon.  Since $\beta\vee\gamma=a\vee \gamma=a$, we conclude again from Remarks~\ref{newremarkCC}\,(2) and Result~\ref{result} that every quantale on $S$ with unit $e$ is an extension of a quantale  by an isolated unit on  the pentagon ${\mathsf{N}}_5$ satisfying (\ref{property(a)}). 

 We can summarize the previous observations as follows.
 
  \begin{corollary} \label{newcor2} Every unitally non-distributive quantale on a complete lattice with $7$ elements is the extension of a quantale provided with property \textup(\ref{property(a)}\textup)  either on the diamond ${\mathsf{M}}_3$ or on the pentagon ${\mathsf{N}}_5$.
\end{corollary}

This is the reason why we are now interested in providing the classification of all quantales satisfying property (\ref{property(a)}) on the diamond ${\mathsf{M}}_3$ and on the pentagon ${\mathsf{N}}_5$. For the convenience of the reader we recall the notation
\[ {\mathsf{M}}=\set{\bot,\alpha,\beta, \gamma,\top} \quad \text{and}\quad {\mathsf{N}}_5=\set{ \bot,\alpha,\beta,\gamma,\top} \qquad \text{(cf.\ (\ref{n3.1C}) {\rm and} (\ref{S71}))}\]
and  start with the following:

\begin{lemma} \label{lemma4.4} Let $\alg{Q}$ be a quantale on the diamond ${\mathsf{M}}_3$ or on the pentagon ${\mathsf{N}}_5$ with $\alpha$, $\beta$ and $\gamma$ as in Figures~\textup(\ref{n3.1C}\textup) and \textup(\ref{S71}\textup). If $\gamma\in \alg{Q}$ satisfies the property \textup(\ref{property(a)}\textup), then $\alg{Q}$ is unital if and only if  $\gamma\ast \gamma=\gamma$.
\end{lemma}

\begin{proof} We first note that if $\alg{Q}=(\alg{Q},\ast,\delta)$ is a unital quantale with unit $\delta$, then (\ref{property(a)}) implies $\gamma\le \delta$ --- i.e. $\delta\in\set{\gamma,\top}$ (see Remarks~\ref{newremarkCC}\,(1)). Moreover, if $\delta=\top$, then the lattice structure of ${\mathsf{M}}_3$ and ${\mathsf{N}}_5$ implies
$\gamma=\top\ast\gamma=(\alpha\ast\gamma)\vee(\beta\ast\gamma)\le {(\alpha\wedge\gamma)\vee(\beta\wedge\gamma)}\le \alpha$, 
which  is a contradiction to $\gamma\not\le \alpha$. 
Hence $\gamma$ is necessarily the unit of $\alg{Q}$ and, in particular,
$\gamma\ast \gamma=\gamma$. Conversely, if $\gamma\ast \gamma=\gamma$, then 
$\gamma\le\top\ast \gamma= (\alpha\ast \gamma)\vee (\beta\ast \gamma)$ and $\gamma\le\gamma\ast \top= (\gamma\ast \alpha)\vee (\gamma\ast \beta)$. Hence (\ref{property(a)}) implies $\alpha\ast \gamma\neq \bot$, $\beta\ast \gamma\neq \bot$, $\gamma\ast \alpha\neq \bot$ and  $\gamma\ast \beta \neq\bot$. Since $\alpha$ and $\beta$ are atoms, we use again (\ref{property(a)}) and conclude that $\gamma$ is the unit of $\alg{Q}$.
\end{proof}

\begin{example}\label{example2} Let $\alg{Q}$ be a quantale on the diamond ${\mathsf{M}}_3=\set{\bot,\alpha,\beta,\gamma,\top}$ and assume that the atom $\gamma$ satisfies (\ref{property(a)}). Then we distinguish the following cases:
\vskip3pt

\noindent
{\sf Case 1.} (Unital quantales) Referring to Lemma~\ref{lemma4.4}, it is easily seen that there exist $8$ non-isomorphic quantales  on ${\mathsf{M}}_3$ with unit  $\gamma$ %, namely %\cite[5.2.11,\,5.2.12,\,5.2.19,\, 5.2.22,\, 5.2.23,\,5.2.42,\,5.2.43,\,5.2.44]{Catalogue} 
(see Figure~\ref{fig:M3.1}). 
The unital quantale 5.2.42 is the quantale induced by the cyclic group with $3$  elements (cf.\ Example~\ref{examples1}\,(2)).
\begin{figure}[H] 
         \vskip-0pt
{\setlength\tabcolsep{1pt}\begin{tabular}{llllllll}
{\scriptsize\setlength\tabcolsep{0pt}
         \begin{tabular}{l}
         $\alpha\ast\alpha=\bot$\\
         $\beta\ast\beta=\gamma$\\
         \hbox to 35pt{\cite[5.2.11]{Catalogue}\hfill}\\
         \end{tabular}}
 &
     {\tiny\setlength\tabcolsep{2pt}
     \begin{tabular}{c|c|c|c|c}
$\ast$ &$\alpha$ & $\beta$ & $\gamma$ &  $\top$\\
\hline
$\alpha$ & $\bot$ & $\alpha$ & $\alpha$ &$\alpha$\\
\hline
$\beta$ &$\alpha$  & $\gamma$ & $\beta$ & $\top$\\
\hline
$\gamma$  & $\alpha$ & $\beta$ & $\gamma$ & $\top$\\
\hline
$\top$ &$\alpha$ & $\top$ & $\top$ & $\top$\\
	\end{tabular}}
         &
         {\scriptsize\setlength\tabcolsep{0pt}
         \begin{tabular}{l}
         $\alpha\ast\alpha=\bot$\\
         $\beta\ast\beta=\top$\\
         \hbox to 35pt{\cite[5.2.12]{Catalogue}\hfill}\\
         \end{tabular}}
 &
     {\tiny\setlength\tabcolsep{2pt}\begin{tabular}{c|c|c|c|c}
$\ast$ &$\alpha$ & $\beta$ & $\gamma$ &  $\top$\\
\hline
$\alpha$ & $\bot$ & $\alpha$ & $\alpha$ &$\alpha$\\
\hline
$\beta$ &$\alpha$  & $\top$ & $\beta$ & $\top$\\
\hline
$\gamma$  & $\alpha$ & $\beta$ & $\gamma$ & $\top$\\
\hline
$\top$ &$\alpha$ & $\top$ & $\top$ & $\top$\\
\end{tabular}}
         &
         {\scriptsize\setlength\tabcolsep{0pt}
         \begin{tabular}{l}
         $\alpha\ast\alpha=\alpha$\\
         $\beta\ast\beta=\beta$\\
         \hbox to 35pt{\cite[5.2.19]{Catalogue}\hfill}\\
         \end{tabular}}
 &
     {\tiny\setlength\tabcolsep{2pt}\begin{tabular}{c|c|c|c|c}
$\ast$ &$\alpha$ & $\beta$ & $\gamma$ &  $\top$\\
\hline
$\alpha$ & $\alpha$ & $\bot$ & $\alpha$ &$\alpha$\\
\hline
$\beta$ &$\bot$  & $\beta$ & $\beta$ & $\top$\\
\hline
$\gamma$  & $\alpha$ & $\beta$ & $\gamma$ & $\top$\\
\hline
$\top$ &$\alpha$ & $\top$ & $\top$ & $\top$\\
\end{tabular}}
       &
         {\scriptsize\setlength\tabcolsep{0pt}
         \begin{tabular}{l}
         $\alpha\ast\alpha=\alpha$\\
          $\beta\ast\beta=\gamma$\\
         \hbox to 35pt{\cite[5.2.22]{Catalogue}\hfill}\\
          \end{tabular}}
 &
     {\tiny\setlength\tabcolsep{2pt}\begin{tabular}{c|c|c|c|c}
$\ast$ &$\alpha$ & $\beta$ & $\gamma$ &  $\top$\\
\hline
$\alpha$ & $\alpha$ & $\alpha$ & $\alpha$ &$\alpha$\\
\hline
$\beta$ &$\alpha$  & $\gamma$ & $\beta$ & $\top$\\
\hline
$\gamma$  & $\alpha$ & $\beta$ & $\gamma$ & $\top$\\
\hline
$\top$ &$\alpha$ & $\top$ & $\top$ & $\top$\\
\end{tabular}}
\\[25pt]
\noalign{\smallskip}
         {\scriptsize\setlength\tabcolsep{0pt}
         \begin{tabular}{l}
         $\alpha\ast\alpha=\alpha$\\
          $\beta\ast\beta=\top$\\
         \hbox to 35pt{\cite[5.2.23]{Catalogue}\hfill}\\
          \end{tabular}}
 &
     {\tiny\setlength\tabcolsep{2pt}\begin{tabular}{c|c|c|c|c}
$\ast$ &$\alpha$ & $\beta$ & $\gamma$ &  $\top$\\
\hline
$\alpha$ & $\alpha$ & $\alpha$ & $\alpha$ &$\alpha$\\
\hline
$\beta$ &$\alpha$  & $\top$ & $\beta$ & $\top$\\
\hline
$\gamma$  & $\alpha$ & $\beta$ & $\gamma$ & $\top$\\
\hline
$\top$ &$\alpha$ & $\top$ & $\top$ & $\top$\\
\end{tabular}}
         &
         {\scriptsize\setlength\tabcolsep{0pt}
         \begin{tabular}{l}
         $\alpha\ast\alpha=\beta$\\
          $\beta\ast\beta=\alpha$\\
         \hbox to 35pt{\cite[5.2.42]{Catalogue}\hfill}\\
         \end{tabular}}
 &
     {\tiny\setlength\tabcolsep{2pt}\begin{tabular}{c|c|c|c|c}
$\ast$ &$\alpha$ & $\beta$ & $\gamma$ &  $\top$\\
\hline
$\alpha$ & $\beta$ & $\gamma$ & $\alpha$ &$\top$\\
\hline
$\beta$ &$\gamma$  & $\alpha$ & $\beta$ & $\top$\\
\hline
$\gamma$  & $\alpha$ & $\beta$ & $\gamma$ & $\top$\\
\hline
$\top$ &$\top$ & $\top$ & $\top$ & $\top$\\
 \end{tabular}}
         &
         {\scriptsize\setlength\tabcolsep{0pt}\begin{tabular}{l}
         $\alpha\ast\alpha=\beta$\\
          $\beta\ast\beta=\top$\\
         \hbox to 35pt{\cite[5.2.43]{Catalogue}\hfill}\\
         \end{tabular}}
 &
     {\tiny\setlength\tabcolsep{2pt}\begin{tabular}{c|c|c|c|c}
$\ast$ &$\alpha$ & $\beta$ & $\gamma$ &  $\top$\\
\hline
$\alpha$ & $\beta$ & $\top$ & $\alpha$ &$\top$\\
\hline
$\beta$ &$\top$  & $\top$ & $\beta$ & $\top$\\
\hline
$\gamma$  & $\alpha$ & $\beta$ & $\gamma$ & $\top$\\
\hline
$\top$ &$\top$ & $\top$ & $\top$ & $\top$\\
\end{tabular}}
&
         \hskip-2pt
         {\scriptsize\setlength\tabcolsep{0pt}
         \begin{tabular}{l}
         $\alpha\ast\alpha=\alpha$\\
          $\beta\ast\beta=\gamma$\\
         \hbox to 35pt{\cite[5.2.22]{Catalogue}\hfill}\\
          \end{tabular}}
 &
     {\tiny\setlength\tabcolsep{2pt}\begin{tabular}{c|c|c|c|c}
$\ast$ &$\alpha$ & $\beta$ & $\gamma$ &  $\top$\\
\hline
$\alpha$ & $\alpha$ & $\alpha$ & $\alpha$ &$\alpha$\\
\hline
$\beta$ &$\alpha$  & $\gamma$ & $\beta$ & $\top$\\
\hline
$\gamma$  & $\alpha$ & $\beta$ & $\gamma$ & $\top$\\
\hline
$\top$ &$\alpha$ & $\top$ & $\top$ & $\top$\\
\end{tabular}}
\end{tabular}
}
\vskip-5pt 
 \caption{${\mathsf{M}}_3$ {\sf Case 1.}  (Unital quantales) $\gamma\ast \gamma=\gamma$. }
          \label{fig:M3.1}
         \vskip-5pt
        \end{figure}
         
\noindent{\sf Case 2.}  (Non-unital quantales) If $\alg{Q}$ is non-unital, then $\gamma\ast \gamma=\bot$. Hence 
\[\top\ast\gamma={(\beta\vee \gamma)\ast \gamma}= \beta \ast \gamma\le \beta\quad\text{and}\quad\top\ast\gamma=(\alpha\vee \gamma)\ast \gamma= \alpha \ast \gamma\le \alpha,\]
 and consequently $\top\ast \gamma=\bot$ follows. Analogously we can prove $\gamma \ast \top=\bot$. If now $x\in \set{\alpha,\beta}$, then we observe: 
 \[\top\ast \top=x\ast x= x\ast \top=\top\ast x=\alpha\ast \beta=\beta\ast \alpha.\]
  Hence the restriction of the quantale multiplication to $\set{\alpha,\beta,\top}$ is constant and is represented by an unique element of ${\mathsf{M}}_3$. Therefore there exist $4$ non-isomorphic and non-unital quantales on ${\mathsf{M}}_3$ satisfying (\ref{property(a)}) %(cf.\ \cite[5.2.1 -- 5.2.4]{Catalogue}) 
(see Figure~\ref{fig:M3.2}).
 Among them there is the trivial quantale on ${\mathsf{M}}_3$ (cf.\ Example~\ref{examples1}\,(1)).
\begin{figure}[H] 
         \vskip-0pt
\centering
{\setlength\tabcolsep{1pt}\begin{tabular}{llllllll}
        {\scriptsize\setlength\tabcolsep{3pt}
         \begin{tabular}{l}
         $\alpha\ast\alpha=\bot$\\
         \cite[5.2.1]{Catalogue}\\
         \end{tabular}}
 &
{\tiny\setlength\tabcolsep{2pt}\begin{tabular}{c|c|c|c|c}
$\ast$ &$\alpha$ & $\beta$ & $\gamma$ &  $\top$\\
\hline
$\alpha$ & $\bot$ & $\bot$ & $\bot$ &$\bot$\\
\hline
$\beta$ &$\bot$  & $\bot$ & $\bot$ & $\bot$\\
\hline
$\gamma$  & $\bot$ & $\bot$ & $\bot$ & $\bot$\\
\hline
$\top$ &$\bot$ & $\bot$ & $\bot$ & $\bot$\\
\end{tabular}}
&
         {\scriptsize\setlength\tabcolsep{3pt}\begin{tabular}{l}
         $\alpha\ast\alpha=\alpha$\\
         \cite[5.2.3]{Catalogue}\\
         \end{tabular}}  
 &
           {\tiny\setlength\tabcolsep{2pt}
           \begin{tabular}{c|c|c|c|c}
$\ast$ &$\alpha$ & $\beta$ & $\gamma$ &  $\top$\\
\hline
$\alpha$ & $\alpha$ & $\alpha$ & $\bot$ &$\alpha$\\
\hline
$\beta$ &$\alpha$  & $\alpha$ & $\bot$ & $\alpha$\\
\hline
$\gamma$  & $\bot$ & $\bot$ & $\bot$ & $\bot$\\
\hline
$\top$ &$\alpha$ & $\alpha$ & $\bot$ & $\alpha$\\
\end{tabular}}
&
         {\scriptsize\setlength\tabcolsep{3pt}\begin{tabular}{l}
         $\alpha\ast\alpha=\gamma$\\
         \cite[5.2.2]{Catalogue}\\
         \end{tabular}}
 &
         {\tiny\setlength\tabcolsep{2pt}\begin{tabular}{c|c|c|c|c}
$\ast$ &$\alpha$ & $\beta$ & $\gamma$ &  $\top$\\
\hline
$\alpha$ & $\gamma$ & $\gamma$ & $\bot$ &$\gamma$\\
\hline
$\beta$ &$\gamma$  & $\gamma$ & $\bot$ & $\gamma$\\
\hline
$\gamma$  & $\bot$ & $\bot$ & $\bot$ & $\bot$\\
\hline
$\top$ &$\gamma$ & $\gamma$ & $\bot$ & $\gamma$\\
\end{tabular}}
&
{\scriptsize\setlength\tabcolsep{3pt}\begin{tabular}{l}
         $\alpha\ast\alpha=\top$\\
         \cite[5.2.4]{Catalogue}\\
         \end{tabular}}
 &
         {\tiny\setlength\tabcolsep{2pt}\begin{tabular}{c|c|c|c|c}
$\ast$ &$\alpha$ & $\beta$ & $\gamma$ &  $\top$\\
\hline
$\alpha$ & $\top$ & $\top$ & $\bot$ &$\top$\\
\hline
$\beta$ &$\top$  & $\top$ & $\bot$ & $\top$\\
\hline
$\gamma$  & $\bot$ & $\bot$ & $\bot$ & $\bot$\\
\hline
$\top$ &$\top$ & $\top$ & $\bot$ & $\top$\\
\end{tabular}
}
\end{tabular}}
\vskip-5pt 
\caption{${\mathsf{M}}_3$ {\sf Case 2.}  (Non-unital quantales) $\gamma\ast \gamma=\bot$. }
          \label{fig:M3.2}
         \vskip-5pt
         \end{figure}

As a concluding remark we point out that  there exist $12$ non-isomorphic, unitally nondistributive quantales on the extended diamond.
\end{example}

\begin{example}\label{example3} Let $\alg{Q}$ be a quantale on the pentagon $\mathsf{N}_5=\set{\bot,\alpha,\beta,\gamma,\top}$ with $\alpha\le \gamma$, and $\alpha\vee \beta=\top$ such that $\gamma$ satisfies (\ref{property(a)}). In particular, $\alpha$ and $\beta$ are atoms and $\alpha \neq \gamma$ holds. Then we distinguish the following cases:
\vskip3pt

\noindent
{\sf Case 1.} (Unital quantales) Referring again to Lemma~\ref{lemma4.4}, it is easily seen that there exist $5$ non-isomorphic quantales  on ${\mathsf{N}}_5$ with unit  $\gamma$ %, namely \cite[5.3.17,\, 5.3.35,\,5.3.42,\, 5.3.184,\, 5.3.229]{Catalogue} 
(see Figure~\ref{fig:N5.1}).
\begin{figure}[H] 
         \vskip-5pt
\centering
{\setlength\tabcolsep{3pt}\begin{tabular}{llllll}
{\scriptsize\setlength\tabcolsep{3pt}
         \begin{tabular}{l}
         $\beta\ast\beta=\bot$\\
         \cite[5.3.17]{Catalogue}\\
         \end{tabular}}
&
         {\scriptsize\setlength\tabcolsep{3pt}\begin{tabular}{c|c|c|c|c}
$\ast$ &$\alpha$ & $\beta$ & $\gamma$ &  $\top$\\
\hline
$\alpha$ & $\alpha$ & $\beta$ & $\alpha$ &$\top$\\
\hline
$\beta$ &$\beta$  & $\bot$ & $\beta$ & $\beta$\\
\hline
$\gamma$  & $\alpha$ & $\beta$ & $\gamma$ & $\top$\\
\hline
$\top$ &$\top$ & $\beta$ & $\top$ & $\top$\\
\end{tabular}}
&
         {\scriptsize\setlength\tabcolsep{3pt}\begin{tabular}{l}
         $\beta\ast\beta=\beta$\\
         $\alpha\ast\beta=\bot$\\
         \cite[5.3.35]{Catalogue}
         \end{tabular}}
&
        {\scriptsize\setlength\tabcolsep{2pt}\begin{tabular}{c|c|c|c|c}
$\ast$ &$\alpha$ & $\beta$ & $\gamma$ &  $\top$\\
\hline
$\alpha$ & $\alpha$ & $\bot$ & $\alpha$ &$\alpha$\\
\hline
$\beta$ &$\bot$  & $\beta$ & $\beta$ & $\beta$\\
\hline
$\gamma$  & $\alpha$ & $\beta$ & $\gamma$ & $\top$\\
\hline
$\top$ &$\alpha$ & $\beta$ & $\top$ & $\top$\\
\end{tabular}}
&
         {\scriptsize\setlength\tabcolsep{3pt}\begin{tabular}{l}
         $\beta\ast\beta=\beta$\\
         $\alpha\ast\beta=\beta$\\
         \cite[5.3.42]{Catalogue}\\
         \end{tabular}}
&
         {\scriptsize\setlength\tabcolsep{2pt}\begin{tabular}{c|c|c|c|c}
$\ast$ &$\alpha$ & $\beta$ & $\gamma$ &  $\top$\\
\hline
$\alpha$ & $\alpha$ & $\beta$ & $\alpha$ &$\top$\\
\hline
$\beta$ &$\beta$  & $\beta$ & $\beta$ & $\beta$\\
\hline
$\gamma$  & $\alpha$ & $\beta$ & $\gamma$ & $\top$\\
\hline
$\top$ &$\top$ & $\beta$ & $\top$ & $\top$\\
\end{tabular}}
\end{tabular}}
\vskip5pt
{\setlength\tabcolsep{3pt}\begin{tabular}{llll}
{\scriptsize\setlength\tabcolsep{3pt}\begin{tabular}{l}
         $\beta\ast\beta=\alpha$\\
         \cite[5.3.184]{Catalogue}\\
         \end{tabular}}
&
         {\scriptsize\setlength\tabcolsep{2pt}\begin{tabular}{c|c|c|c|c}
$\ast$ &$\alpha$ & $\beta$ & $\gamma$ &  $\top$\\
\hline
$\alpha$ & $\alpha$ & $\beta$ & $\alpha$ &$\top$\\
\hline
$\beta$ &$\beta$  & $\alpha$ & $\beta$ & $\top$\\
\hline
$\gamma$  & $\alpha$ & $\beta$ & $\gamma$ & $\top$\\
\hline
$\top$ &$\top$ & $\top$ & $\top$ & $\top$\\
\end{tabular}}
&
         {\scriptsize\setlength\tabcolsep{3pt}\begin{tabular}{l}
         $\beta\ast\beta=\top$\\
         \cite[5.3.229]{Catalogue}
         \end{tabular}}
&
         {\scriptsize\setlength\tabcolsep{2pt}\begin{tabular}{c|c|c|c|c}
$\ast$ &$\alpha$ & $\beta$ & $\gamma$ &  $\top$\\
\hline
$\alpha$ & $\alpha$ & $\beta$ & $\alpha$ &$\top$\\
\hline
$\beta$ &$\beta$  & $\top$ & $\beta$ & $\top$\\
\hline
$\gamma$  & $\alpha$ & $\beta$ & $\gamma$ & $\top$\\
\hline
$\top$ &$\top$ & $\top$ & $\top$ & $\top$\\
\end{tabular}}
\end{tabular}}
\vskip-5pt 
                  \caption{${\mathsf{N}}_5$ {\sf Case 1.}  (Unital) $\gamma\ast \gamma=\gamma$. }
          \label{fig:N5.1}
         \vskip-5pt
         \end{figure}
         
     \noindent {\sf Case 2.} (Non-unital) First we show that a non-unital quantale $\alg{Q}$ on $\mathsf{N}_5$ satisfying (\ref{property(a)})  is semi-unital if and only if the following condition holds:
\stepcounter{num}
\begin{equation} \label{n3.3} \gamma\ast \gamma=\alpha\quad \text{and}\quad \beta\ast \gamma=\beta=\gamma\ast \beta.
\end{equation}
Since $\alg{Q}$ is not unital, we conclude from the previous case that $\gamma\ast \gamma\le \alpha$ holds. Now let $\alg{Q}$ be semi-unital.  If we would assume $\gamma\ast \gamma=\bot$, then we would obtain the contradiction $\gamma\le \top\ast \gamma=(\beta\vee \gamma)\ast \gamma=\beta\ast \gamma\le \beta$. Hence $\gamma\ast \gamma=\alpha$ follows. Further, $\gamma\le \top\ast \gamma= (\beta\ast \gamma)\vee (\alpha\ast \gamma)$ and $\gamma\not\le \alpha$ imply $\beta\ast \gamma\neq \bot$. Since $\beta$ is an atom, the relation $\beta\ast \gamma=\beta$ follows. Analogously we verify $\gamma\ast \beta=\beta$.\\
On the other hand, if (\ref{n3.3}) holds, then we first observe $\alpha\ast \alpha=\alpha$. Otherwise,  if we would assume $\alpha\ast \alpha=\bot$, then we would obtain the  contradiction $\bot=\beta\ast \alpha\ast \alpha= \beta\ast \gamma\ast \gamma\ast \gamma\ast \gamma=\beta$.
Therefore it is easily seen that 
\[\alpha=\alpha\ast \alpha\le \top\ast \alpha, \quad\beta=\gamma \ast \beta\le \top\ast \beta\quad\text{and}\quad\gamma\le \alpha\vee \beta=(\gamma\ast \gamma)\vee (\beta\ast \gamma)=\top\ast \gamma.\]
Analogously we verify $\alpha\le \alpha\ast \top,\,\beta\le \beta\ast \top$ and $\gamma \ast \top =\top$.
\vskip3pt

\noindent
{\sf Case 2.1.}(Non-unital and semi-unital) The relations (\ref{n3.3}) and (\ref{property(a)}) imply $\alpha\ast \alpha={\alpha\ast \gamma}=\gamma\ast \alpha=\alpha$ and $\alpha\ast \beta=\beta \ast \alpha=\beta$. Further we obtain 
\[\top\ast \alpha= (\beta\ast \alpha)\vee (\alpha\ast \alpha)=\beta\vee \alpha=\top,\quad\top\ast \beta= (\beta\ast \beta)\vee \beta\quad\text{and}\quad\top\ast \gamma=\top.\]
Analogously the relations $\alpha\ast \top =\top$, $ \beta\ast \top= (\beta\ast \beta)\vee\beta$ and $\gamma\ast \top =\top$ hold. Finally, if we would assume $\beta\ast \beta= \gamma$, then we would obtain the contradiction $\alpha=\gamma\ast \alpha= \beta\ast \beta\ast \alpha= \beta\ast \beta=\gamma$.
Hence $ \beta\ast \beta\in \set{\bot,\beta,\alpha,\top}$, and so we have  $4$ non-isomorphic, non-unital, but semi-unital  quantales on $\mathsf{N}_5$ satisfying (\ref{property(a)}) %, namely \cite[5.3.16, 5.3.41, 5.3.183, 5.3.228]{Catalogue} 
(see Figure~\ref{fig:N5.2}).
\begin{figure}[H] 
         \vskip-5pt
\centering
{\setlength\tabcolsep{1pt}\begin{tabular}{llllllll}
        {\scriptsize
         \setlength\tabcolsep{3pt}
         \begin{tabular}{l}
         $\beta\ast\beta=\bot$\\
         \cite[5.3.16]{Catalogue}\\
         \end{tabular}}
 &
         {\tiny\setlength\tabcolsep{2pt}\begin{tabular}{c|c|c|c|c}
$\ast$ &$\alpha$ & $\beta$ & $\gamma$ &  $\top$\\
\hline
$\alpha$ & $\alpha$ & $\beta$ & $\alpha$ &$\top$\\
\hline
$\beta$ &$\beta$  & $\bot$ & $\beta$ & $\beta$\\
\hline
$\gamma$  & $\alpha$ & $\beta$ & $\alpha$ & $\top$\\
\hline
$\top$ &$\top$ & $\beta$ & $\top$ & $\top$\\
\end{tabular}}
 &
         {\scriptsize\setlength\tabcolsep{3pt}
         \begin{tabular}{l}
         $\beta\ast\beta=\alpha$\\
         \cite[5.3.183]{Catalogue}\\
         \end{tabular}}
 &
         {\tiny\setlength\tabcolsep{2pt}\begin{tabular}{c|c|c|c|c}
$\ast$ &$\alpha$ & $\beta$ & $\gamma$ &  $\top$\\
\hline
$\alpha$ & $\alpha$ & $\beta$ & $\alpha$ &$\top$\\
\hline
$\beta$ &$\beta$  & $\alpha$ & $\beta$ & $\top$\\
\hline
$\gamma$  & $\alpha$ & $\beta$ & $\alpha$ & $\top$\\
\hline
$\top$ &$\top$ & $\top$ & $\top$ & $\top$\\
\end{tabular}}
 &
         {\scriptsize\setlength\tabcolsep{3pt}
         \begin{tabular}{l}
         $\beta\ast\beta=\beta$\\
         \cite[5.3.41]{Catalogue}\\
         \end{tabular}}
 &
          {\tiny\setlength\tabcolsep{2pt}\begin{tabular}{c|c|c|c|c}
$\ast$ &$\alpha$ & $\beta$ & $\gamma$ &  $\top$\\
\hline
$\alpha$ & $\alpha$ & $\beta$ & $\alpha$ &$\top$\\
\hline
$\beta$ &$\beta$  & $\beta$ & $\beta$ & $\beta$\\
\hline
$\gamma$  & $\alpha$ & $\beta$ & $\alpha$ & $\top$\\
\hline
$\top$ &$\top$ & $\beta$ & $\top$ & $\top$\\
\end{tabular}}
 &
         {\scriptsize\setlength\tabcolsep{3pt}\begin{tabular}{l}
         $\beta\ast\beta=\top$\\
         \cite[5.3.228]{Catalogue}\\
         \end{tabular}}
 &
         {\tiny\setlength\tabcolsep{2pt}\begin{tabular}{c|c|c|c|c}
$\ast$ &$\alpha$ & $\beta$ & $\gamma$ &  $\top$\\
\hline
$\alpha$ & $\alpha$ & $\beta$ & $\alpha$ &$\top$\\
\hline
$\beta$ &$\beta$  & $\top$ & $\beta$ & $\top$\\
\hline
$\gamma$  & $\alpha$ & $\beta$ & $\alpha$ & $\top$\\
\hline
$\top$ &$\top$ & $\top$ & $\top$ & $\top$\\
\end{tabular}}
\end{tabular}}
%\vskip-5pt 
                      \caption{${\mathsf{N}}_5$ {\sf Case 2.1} (Non-unital, but semi-unital quantales) $\gamma\ast \gamma=\alpha$.  $\beta\ast \gamma=\beta=\gamma\ast \beta$.  }
          \label{fig:N5.2}
         \vskip-5pt
         \end{figure}
\pagebreak

\noindent{\sf Case 2.2.} (Non semi-unital quantales) With regard to (\ref{n3.3}) we distinguish two cases:
\vskip3pt

\noindent
{\sf Case 2.2.1} $\gamma\ast \gamma=\alpha$ and $(\beta\ast \gamma)\wedge(\gamma \ast \beta)=\bot$.  Depending on the possible values of $\beta\ast \gamma,\gamma \ast \beta\in\set{\bot,\beta}$ we obtain $4$ non-isomorphic quantales  on ${\mathsf{N}}_5$,  the last $2$ of them being non-commutative (see Figure~\ref{fig:N5.3.1}). \\
\noindent If  $\beta\ast \gamma=\bot=\gamma \ast \beta$, then $\beta\ast \alpha=\bot=\alpha\ast \beta$. Further, $\gamma\ast \gamma=\alpha$ implies $\alpha\ast\gamma\neq\bot$ and $\gamma\ast \alpha\neq \bot$. Hence $\alpha\ast \gamma=\gamma\ast \alpha=\alpha=\alpha\ast \alpha$ follows. Since $\gamma\ast \gamma=\alpha$ and $\gamma\ast \alpha=\alpha$, we conclude $\beta\ast \beta\in \set{\bot,\beta}$, and so we obtain the first $2$ quantales in Figure~\ref{fig:N5.3.1}. If $\beta\ast \gamma=\beta$ and $\gamma\ast \beta=\bot$ then $\beta\ast \alpha=\beta$ and $\alpha\ast \beta=\bot$. In particular $\beta\ast \beta=\bot$ follows, and consequently we have $\top\ast \beta=\bot$. Further, if we would assume $\gamma\ast \alpha= \bot$, then we would obtain the  contradiction $\alpha=\gamma\ast \gamma\le \gamma\ast (\beta\vee \alpha)=\bot$. Hence $\gamma\ast \alpha=\alpha$ follows, and $\gamma\ast \gamma=\alpha$ implies $\alpha\ast \alpha=\alpha$, and consequently $\alpha\ast \gamma=\alpha$. Analoguously we can treat the case $\beta\ast \gamma=\bot$ and $\gamma\ast \beta=\beta$ and arrive at the opposite quantale. In this way we obtain the last $2$ (non-commutative) quantales  in Figure~\ref{fig:N5.3.1}.
\begin{figure}[H] 
         \vskip-5pt
\centering
{\setlength\tabcolsep{1pt}\begin{tabular}{llllllll}
        {\scriptsize\setlength\tabcolsep{3pt}
 \begin{tabular}{l}
         $\beta\ast \gamma=\bot$\\
         $\gamma\ast \beta=\bot$\\
         \cite[5.3.3]{Catalogue}
         \end{tabular}}
&
        {\tiny\setlength\tabcolsep{2pt}\begin{tabular}{c|c|c|c|c}
$\ast$ &$\alpha$ & $\beta$ & $\gamma$ &  $\top$\\
\hline
$\alpha$ & $\alpha$ & $\bot$ & $\alpha$ &$\alpha$\\
\hline
$\beta$ &$\bot$  & $\bot$ & $\bot$ & $\bot$\\
\hline
$\gamma$  & $\alpha$ & $\bot$ & $\alpha$ & $\alpha$\\
\hline
$\top$ &$\alpha$ & $\bot$ & $\alpha$ & $\alpha$\\
\end{tabular}}
&
         {\scriptsize\setlength\tabcolsep{3pt}\begin{tabular}{l}
         $\beta\ast \gamma=\bot$\\
         $\gamma\ast \beta=\bot$\\
         \cite[5.3.29]{Catalogue}
         \end{tabular}}
&
        {\tiny\setlength\tabcolsep{2pt}\begin{tabular}{c|c|c|c|c}
$\ast$ &$\alpha$ & $\beta$ & $\gamma$ &  $\top$\\
\hline
$\alpha$ & $\alpha$ & $\bot$ & $\alpha$ &$\alpha$\\
\hline
$\beta$ &$\bot$  & $\beta$ & $\bot$ & $\beta$\\
\hline
$\gamma$  & $\alpha$ & $\bot$ & $\alpha$ & $\alpha$\\
\hline
$\top$ &$\alpha$ & $\beta$ & $\alpha$ & $\top$\\
\end{tabular}}
&
         {\scriptsize\setlength\tabcolsep{3pt}\begin{tabular}{l}
         $\beta\ast \gamma=\beta$\\
         $\gamma\ast \beta=\bot$\\
         \cite[5.3.11]{Catalogue}
         \end{tabular}}
&
       {\tiny\setlength\tabcolsep{2pt}\begin{tabular}{c|c|c|c|c}
$\ast$ &$\alpha$ & $\beta$ & $\gamma$ &  $\top$\\
\hline
$\alpha$ & $\alpha$ & $\beta$ & $\alpha$ &$\top$\\
\hline
$\beta$ &$\bot$  & $\bot$ & $\bot$ & $\bot$\\
\hline
$\gamma$  & $\alpha$ & $\beta$ & $\alpha$ & $\top$\\
\hline
$\top$ &$\alpha$ & $\beta$ & $\alpha$ & $\top$\\
\end{tabular}}
&
         {\scriptsize\setlength\tabcolsep{3pt}
         \begin{tabular}{l}
         $\beta\ast \gamma=\bot$\\
         $\gamma\ast \beta=\beta$\\
         \cite[5.3.6]{Catalogue}
         \end{tabular}}
&
        {\tiny\setlength\tabcolsep{2pt}\begin{tabular}{c|c|c|c|c}
$\ast$ &$\alpha$ & $\beta$ & $\gamma$ &  $\top$\\
\hline
$\alpha$ & $\alpha$ & $\bot$ & $\alpha$ &$\alpha$\\
\hline
$\beta$ &$\beta$  & $\bot$ & $\beta$ & $\beta$\\
\hline
$\gamma$  & $\alpha$ & $\bot$ & $\alpha$ & $\alpha$\\
\hline
$\top$ &$\top$ & $\bot$ & $\top$ & $\top$\\
\end{tabular}}
\end{tabular}}
\vskip-5pt 
                 \caption{${\mathsf{N}}_5$ {\sf Case 2.2.1}  (Non semi-unital quantales)  $\gamma\ast \gamma=\alpha$.\\ $(\beta\ast \gamma)\wedge(\gamma \ast \beta)=\bot$.}
          \label{fig:N5.3.1}
         \vskip-5pt
         \end{figure}

    \noindent{\sf Case 2.2.2} $\gamma\ast \gamma=\bot$.  Then $\alpha\ast\alpha=\bot$ and $\gamma\ast \beta=\beta\ast \gamma=\bot$, and consequently 
    \[\alpha\ast \beta=\beta \ast \alpha=\alpha\ast \gamma=\gamma\ast \alpha=\bot.\]
     Hence $\beta\ast \beta=\beta\ast \top=\top\ast \beta=\top\ast \top$ and $\beta\ast \beta$ can attain all elements of $\mathsf{N}_5$. 
Therefore there exist $5$ non-isomorphic and non-unital quantales on ${\mathsf{N}}_5$ satisfying (\ref{property(a)}) %(cf. \cite[5.3.1, 5.3.28, 5.3.178, 5.3.189, 5.3.207]{Catalogue}).
 (see Figure~\ref{fig:N5.3.2}). 
Among them there is the trivial quantale on ${\mathsf{N}}_5$ (cf.\ Example~\ref{examples1}\,(1)).
\begin{figure}[H] 
         \vskip-5pt
\centering
{\setlength\tabcolsep{1pt}\begin{tabular}{llllll}
        {\scriptsize\setlength\tabcolsep{3pt}
         \begin{tabular}{l}
         $\beta\ast\beta=\bot$\\
         \cite[5.3.1]{Catalogue}\\
         \end{tabular}}
&
         {\scriptsize\setlength\tabcolsep{2pt}\begin{tabular}{c|c|c|c|c}
$\ast$ &$\alpha$ & $\beta$ & $\gamma$ &  $\top$\\
\hline
$\alpha$ & $\bot$ & $\bot$ & $\bot$ &$\bot$\\
\hline
$\beta$ &$\bot$  & $\bot$ & $\bot$ & $\bot$\\
\hline
$\gamma$  & $\bot$ & $\bot$ & $\bot$ & $\bot$\\
\hline
$\top$ &$\bot$ & $\bot$ & $\bot$ & $\bot$\\
\end{tabular}} 
&
         {\scriptsize\setlength\tabcolsep{3pt}\begin{tabular}{l}
         $\beta\ast\beta=\alpha$\\
         \cite[5.3.178]{Catalogue}\\
         \end{tabular}}
&
         {\scriptsize\setlength\tabcolsep{2pt}\begin{tabular}{c|c|c|c|c}
$\ast$ &$\alpha$ & $\beta$ & $\gamma$ &  $\top$\\
\hline
$\alpha$ & $\bot$ & $\bot$ & $\bot$ &$\bot$\\
\hline
$\beta$ &$\bot$  & $\alpha$ & $\bot$ & $\alpha$\\
\hline
$\gamma$  & $\bot$ & $\bot$ & $\bot$ & $\bot$\\
\hline
$\top$ &$\bot$ & $\alpha$ & $\bot$ & $\alpha$\\
\end{tabular}}
&
         {\scriptsize\setlength\tabcolsep{3pt}\begin{tabular}{l}
         $\beta\ast\beta=\beta$\\
         \cite[5.3.28]{Catalogue}\\
          \end{tabular}}
&
         {\scriptsize\setlength\tabcolsep{2pt}\begin{tabular}{c|c|c|c|c}
$\ast$ &$\alpha$ & $\beta$ & $\gamma$ &  $\top$\\
\hline
$\alpha$ & $\bot$ & $\bot$ & $\bot$ &$\bot$\\
\hline
$\beta$ &$\bot$  & $\beta$ & $\bot$ & $\beta$\\
\hline
$\gamma$  & $\bot$ & $\bot$ & $\bot$ & $\bot$\\
\hline
$\top$ &$\bot$ & $\beta$ & $\bot$ & $\beta$\\
\end{tabular}}
 \end{tabular}}
\vskip5pt
{\setlength\tabcolsep{1pt}\begin{tabular}{llllll}         
{\scriptsize\setlength\tabcolsep{3pt}
         \begin{tabular}{l}
         $\beta\ast\beta=\gamma$\\
         \cite[5.3.189]{Catalogue}\\
         \end{tabular}}
&
         {\scriptsize\setlength\tabcolsep{2pt}\begin{tabular}{c|c|c|c|c}
$\ast$ &$\alpha$ & $\beta$ & $\gamma$ &  $\top$\\
\hline
$\alpha$ & $\bot$ & $\bot$ & $\bot$ &$\bot$\\
\hline
$\beta$ &$\bot$  & $\gamma$ & $\bot$ & $\gamma$\\
\hline
$\gamma$  & $\bot$ & $\bot$ & $\bot$ & $\bot$\\
\hline
$\top$ &$\bot$ & $\gamma$ & $\bot$ & $\gamma$\\
\end{tabular}}
&
         {\scriptsize\setlength\tabcolsep{3pt}
         \begin{tabular}{l}
         $\beta\ast\beta=\top$\\
         \cite[5.3.207]{Catalogue}\\
         \end{tabular}}
         {\scriptsize\setlength\tabcolsep{2pt}\begin{tabular}{c|c|c|c|c}
$\ast$ &$\alpha$ & $\beta$ & $\gamma$ &  $\top$\\
\hline
$\alpha$ & $\bot$ & $\bot$ & $\bot$ &$\bot$\\
\hline
$\beta$ &$\bot$  & $\top$ & $\bot$ & $\top$\\
\hline
$\gamma$  & $\bot$ & $\bot$ & $\bot$ & $\bot$\\
\hline
$\top$ &$\bot$ & $\top$ & $\bot$ & $\top$\\
\end{tabular}}
\end{tabular}}
\vskip-5pt 
         \caption{${\mathsf{N}}_5$ {\sf Case 2.2.2} (Non semi-unital) $\gamma\ast \gamma=\bot$. }
\label{fig:N5.3.2}
         \vskip-5pt
         \end{figure}

Finally, we can summarize the situation as follows. There exist $9$ non-isomorphic and not semi-unital quantales on the pentagon satisfying (\ref{property(a)}). Among them there exist exactly two non-commutative quantales, namely \cite[5.3.6,\, 5.3.11]{Catalogue}.
\end{example}

As a concluding remark of the Examples \ref{example2} and \ref{example3}  we point out (cf.\ Corollary~\ref{newcor1}, Proposition~\ref{newprop3} and Proposition~\ref{newnewprop6}) that on the set of $7$ elements there exist exactly $30$ non-isomorphic, unitally  nondistributive quantales. Among them there are $2$ non-commutative ones.
\pagebreak

\subsection{Unitally nondistributive quantales on a complete lattice with at least $8$ elements}
\label{subsec:3.2}
In a third step we consider the case $ \gamma \neq \alpha\vee \gamma\neq a$ and $\beta \vee \gamma =a$. This situation implies $\alpha\not\le \gamma$ and $\beta\not\le \gamma$, and consequently $\gamma$ is a third atom.   Then the lattice-structure of $S$ is characterized by the  Hasse diagram:
 \stepcounter{num}
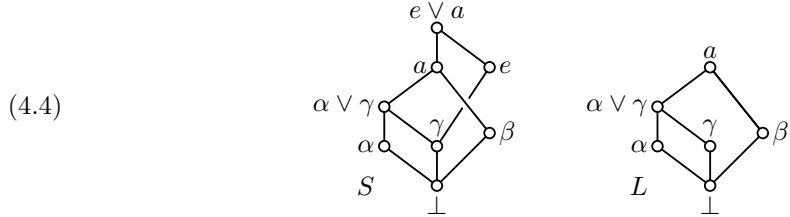
\begin{equation}\label{L8.1}
\begin{tikzpicture}[x=7mm,y=7mm,baseline={([yshift=-.5ex]current bounding box.center)},vertex/.style={anchor=base}]
         \draw (1,0)--(2,1)--(1,2.25)--(1,3);
         \draw (1,0)--(0,0.75)--(0,1.5)--(1,2.25);
         \draw (1,0)--(1,0.75);
         \draw (1,0.75)--(0,1.5);
         \draw (1,0.75)--(2,2.25)--(1,3);
         \draw[white,line width=4pt] (2,1)--(1,2.25);
         \draw (2,1.)--(1,2.25);
         \mycircle{2,2.25};
         \node[right] at (2,2.25) {$e$};
         \mycircle{1,0};
         \node[below] at (1,0) {$\bot$};
         \mycircle{0,0.75};
         \node[left] at (0,0.75) {$\alpha$};
         \mycircle{1,0.75};
         \node[above] at (1,0.75) {$\gamma$};
         \mycircle{2,1};
         \node[right] at (2,1) {$\beta$};
         \mycircle{0,1.5};
         \node[left] at (0,1.5) {$\alpha \vee \gamma$};
         \mycircle{1,2.25};
         \node[left] at (1,2.25) {$a$};
         \mycircle{1,3};
         \node[above] at (1,3) {$e \vee a$}; 
         \node[left] at (0,0) {$S$}; 
         \end{tikzpicture}
         \qquad
\begin{tikzpicture}[x=7mm,y=7mm,baseline={([yshift=-.5ex]current bounding box.center)},vertex/.style={anchor=base}]
         \draw (1,0)--(2,1)--(1,2.25);
         \draw (1,0)--(0,0.75)--(0,1.5)--(1,2.25);
         \draw (1,0)--(1,0.75);
         \draw (1,0.75)--(0,1.5);
         \draw (2,1.)--(1,2.25);
         \mycircle{1,0};
         \node[below] at (1,0) {$\bot$};
         \mycircle{0,0.75};
         \node[left] at (0,0.75) {$\alpha$};
         \mycircle{1,0.75};
         \node[above] at (1,0.75) {$\gamma$};
         \mycircle{2,1};
         \node[right] at (2,1) {$\beta$};
         \mycircle{0,1.5};
         \node[left] at (0,1.5) {$\alpha \vee \gamma$};
         \mycircle{1,2.25};
         \node[above] at (1,2.25) {$a$};
         \node[above,white] at (1,3) {$e \vee a$}; 
         \node[left] at (0,0) {$L$}; 
         \end{tikzpicture}
         \end{equation}
%\vskip-5pt 
\noindent 
and $L$ has the form $L=\mathsf{L}_6$ (cf. Figure~\ref{fig:L6 and L7}). Again we conclude from Result~\ref{result} that every quantale on $S$ with unit $e$ is an extension of a quantale on $\mathsf{L}_6$ satisfying (\ref{property(a)}) and (\ref{property(aa)}) by an isolated unit. 
However, referring to Remarks~\ref{newremarkCC}\,(3) we observe that the algebraic situation on the nondistributive lattice $\mathsf{L}_6$ is fundamentally different from that on the pentagon or on the diamond with three atoms (see Subsection~\ref{subsec:3.1}), since in $\mathsf{M}_3$ and $\mathsf{N}_5$ the condition (3.2) always implies  (3.3) (cf.\ Remark~\ref{newremarkCC}~(2)).

In the last step we consider the general case, namely $\alpha\vee \gamma\neq a$ and $\beta\vee\gamma\neq a$.  Since $\alpha\vee \beta=a$, this situation implies $\alpha\vee \gamma \neq \gamma$ and $\beta\vee \gamma\neq \gamma$. Then  again $\gamma$ is an atom in $S$ and the lattice-theoretic structure of $S$ is characterized by the Hasse diagram:
 \stepcounter{num}
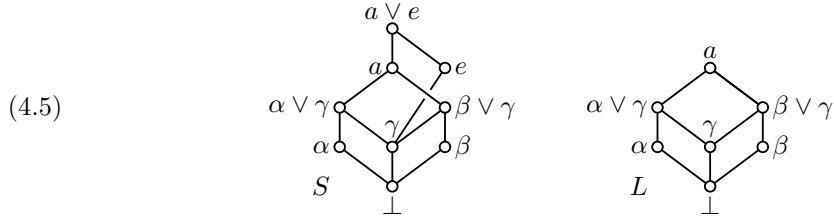
\begin{equation}\label{L9}
\begin{tikzpicture}[x=7mm,y=7mm,baseline={([yshift=-.5ex]current bounding box.center)},vertex/.style={anchor=base}]
         \draw (1,0)--(2,0.75)--(2,1.5)--(1,2.25)--(1,3);
         \draw (1,0)--(0,0.75)--(0,1.5)--(1,2.25);
         \draw (1,0)--(1,0.75)--(0,1.5);
         \draw (1,0.75)--(2,1.5);
         \draw (1,0.75)--(2,2.25)--(1,3);
         \draw[white,line width=4pt] (2,1.5)--(1,2.25);
         \draw (2,1.5)--(1,2.25);
         \mycircle{2,2.25};
         \node[right] at (2,2.25) {$e$};
         \mycircle{1,0};
         \node[below] at (1,0) {$\bot$};
         \mycircle{0,0.75};
         \node[left] at (0,0.75) {$\alpha$};
         \mycircle{1,0.75};
         \node[above] at (1,0.75) {$\gamma$};
         \mycircle{2,0.75};
         \node[right] at (2,0.75) {$\beta$};
         \mycircle{0,1.5};
         \node[left] at (0,1.5) {$\alpha \vee \gamma$};
         \mycircle{2,1.5};
         \node[right] at (2,1.5) {$\beta \vee \gamma$};
         \mycircle{1,2.25};
         \node[left] at (1,2.25) {$a$};
         \mycircle{1,3};
         \node[above] at (1,3) {$a\vee e$}; 
         \node[left] at (0,0) {$S$}; 
         \end{tikzpicture}
         \qquad\begin{tikzpicture}[x=7mm,y=7mm,baseline={([yshift=-.5ex]current bounding box.center)},vertex/.style={anchor=base}]
         \draw (1,0)--(2,0.75)--(2,1.5)--(1,2.25);
         \draw (1,0)--(0,0.75)--(0,1.5)--(1,2.25);
         \draw (1,0)--(1,0.75)--(0,1.5);
        \draw (1,0.75)--(2,1.5)--(1,2.25);
         \mycircle{1,0};
         \node[below] at (1,0) {$\bot$};
         \mycircle{0,0.75};
         \node[left] at (0,0.75) {$\alpha$};
         \mycircle{1,0.75};
         \node[above] at (1,0.75) {$\gamma$};
         \mycircle{2,0.75};
         \node[right] at (2,0.75) {$\beta$};
         \mycircle{0,1.5};
         \node[left] at (0,1.5) {$\alpha \vee \gamma$};
         \mycircle{2,1.5};
         \node[right] at (2,1.5) {$\beta \vee \gamma$};
         \mycircle{1,2.25};
         \node[above] at (1,2.25) {$a$};
         \node[above,white] at (1,3) {$a\vee e$}; 
         \node[left] at (0,0) {$L$}; 
         \end{tikzpicture}
         \end{equation}
%\vskip-5pt 
\noindent 
and $L$ has the form $L=\mathsf{L}_7$ (cf. Figure~\ref{fig:L6 and L7}). Referring to Result~\ref{result} every unital quantale on $S$ with unit $e$ is an extension of a quantale on  $\mathsf{L}_7$   satisfying (\ref{property(a)}) and (\ref{property(aa)}) by an isolated unit.
With regard to further investigations we include the following remark.

\begin{remark}\label{newremark2} Let $\alg{Q}$ be a quantale on $\mathsf{L}_7$. If $\alg{Q}$ satisfies (\ref{property(a)}) and  (\ref{property(aa)}), then $a\ast \gamma=\gamma\ast a= \bot$. In fact, since ${(\alpha\vee \gamma)\wedge (\beta \vee \gamma)}=\gamma$, Lemma~\ref{newlem4}\,\ref{newlem4(4)} and Theorem~\ref{newthm2} imply that $\gamma$ is two-sided in $\overline{\mathsf{L}}_7$, and consequently $a\ast \gamma\le \gamma$ holds. Hence the relation
\[a\ast \gamma=(\alpha\ast \gamma)\vee (\beta\ast \gamma)\le (\alpha\wedge\gamma)\vee (\beta\wedge \gamma)=\bot\]
follows. Analogously we can verify $\gamma \ast a=\bot$.\\
Hence there exists a relationship between unitally nondistributive quantales on $\overline{\mathsf{L}}_7$ and $\overline{\mathsf{M}}_3$ as follows: 
Every unitally nondistribitive quantale on $\overline{\mathsf{M}}_3$ being an extension of the non-unital quantales \cite[5.2.1 --- 5.2.4]{Catalogue} by an isolated unit $e$ is a quotient of a unital quantale on $\overline{\mathsf{L}}_7$ with unit $e$ in the sense of the category of unital quantales. 

We finish this remark with some details of this construction. Since $\top\ast \gamma=\bot=\gamma \ast \top$ holds in all non-unital quantales on the diamond $\mathsf{M}_3$ (cf.\ Example~\ref{example2}), the quantale multiplications in \cite[5.2.1 --- 5.2.4]{Catalogue} have an extension to a quantale multiplication in $\mathsf{L}_7$ by defining:
\[x\ast (\alpha\vee \gamma):=x\ast \alpha,\quad x\ast (\beta \vee \gamma):=x\ast \beta,\quad (\alpha\vee \gamma)\ast x:=\alpha\ast x,\quad (\beta\vee \gamma)\ast x:=\beta\ast x,\]
for each $x\in \mathsf{M}_3$.
In a second step we realize that all these quantales satisfy the properties (\ref{property(a)}) and (\ref{property(aa)}). Hence we can apply Theorem~\ref{newthm2} and consider their extension by an isolated unit on $\overline{\mathsf{L}}_7$. Finally, $\overline{\mathsf{L}}_7\xrightarrow{\,c\,} \overline{\mathsf{L}}_7$ defined by (cf.\ Figure~\ref{fig:L6 and L7}):
\[c(\alpha\vee \gamma)=\top=c(\beta\vee \gamma),\quad c(z)=z\quad \text{for}\,\,\, z\in \overline{\mathsf{L}}_7\setminus \set{\alpha\vee \gamma,\beta\vee \gamma}\]
is a nucleus, and the respective quotients w.r.t.\ $c$ are isomorphic to the extension of the respective non-unital quantales on $\mathsf{M}_3$ by an isolated unit.
\end{remark}

\noindent
{\bf Acknowledgment.}
The argument in Comments~\ref{newnewcomment2}\,(b) was pointed out to us by the referee. We greatly appreciate this comment.

\end{document}